\documentclass[11pt]{amsart}
\usepackage{amsthm,amsmath,amssymb,amsfonts}
\usepackage{verbatim}
\usepackage{pb-diagram}

\newtheorem{thm}{Theorem}
\newtheorem{lem}[thm]{Lemma}
\newtheorem{cor}[thm]{Corollary}
\newtheorem{prop}[thm]{Proposition}

\newtheorem{conj}[thm]{Conjecture}

\theoremstyle{definition}
\newtheorem{definition}[thm]{Definition}
\newtheorem{rem}[thm]{Remark}
\newtheorem{rems}[thm]{Remarks}
\newtheorem{ex}[thm]{Example}

\numberwithin{equation}{section}
\numberwithin{thm}{section}
\numberwithin{figure}{section}

\newcommand{\up}{{\rm up}}
\newcommand{\down}{{\rm dn}}

\newcommand{\al}{\alpha}

\newcommand{\C}{\mathbb{C}}

\newcommand{\Dec}{\mathrm{Dec}}

\newcommand{\End}{\mathrm{End}}
\newcommand{\Frac}{\mathrm{Frac}}
\newcommand{\F}{\mathrm{Fun}}

\newcommand{\geh}{\mathfrak{g}}

\newcommand{\Hom}{\mathrm{Hom}}
\newcommand{\hs}{\widehat{s}}
\newcommand{\Inc}{\mathrm{Inc}}

\newcommand{\ip}[2]{\langle #1\,,\,#2\rangle}

\newcommand{\K}{{\mathbb K}}

\newcommand{\la}{\lambda}
\newcommand{\La}{\Lambda}

\newcommand{\mn}{\mathrm{up}}

\newcommand{\OO}{\mathcal{O}}
\newcommand{\Pcp}{\Phi^{\vee+}}
\newcommand{\Pd}{\mathcal{D}}
\newcommand{\Pu}{\mathcal{U}}

\newcommand{\pnt}{\mathrm{pt}}
\newcommand{\Q}{\mathbb{Q}}
\newcommand{\R}{\mathbb{R}}

\newcommand{\rcht}{\overline{\mathrm{rht}}}
\newcommand{\rht}{\mathrm{rht}}

\newcommand{\res}{\mathrm{res}}
\newcommand{\sgn}{\mathrm{sgn}}

\newcommand{\St}{\mathcal{S}}

\newcommand{\Ta}{\mathcal{T}}

\newcommand{\ts}{\widetilde{s}}

\newcommand{\tw}{\tilde{w}}
\newcommand{\tz}{\tilde{z}}

\newcommand{\UU}{\mathrm{U}}
\newcommand{\wt}{\mathrm{wt}}

\newcommand{\Xc}{X^\circ}

\newcommand{\Z}{{\mathbb Z}}

\begin{document}

\title{Equivariant $K$-Chevalley Rules for Kac-Moody Flag Manifolds}

\author{Cristian Lenart}
\address{Department of Mathematics and Statistics, State University of New York at Albany, Albany, NY 12222}
\email{clenart@albany.edu}

\thanks{C. Lenart was partially supported by the National Science Foundation grant DMS--1101264}

\author{Mark Shimozono}
\address{Department of Mathematics, Virginia Polytechnic Institute
and State University, Blacksburg, VA 24061-0123 USA}
\email{mshimo@vt.edu}

\thanks{M. Shimozono was partially supported by the National Science Foundation grants 
DMS--0652648 and DMS--0652641}

\begin{abstract} Explicit combinatorial cancellation-free rules are given for the
product of an equivariant line bundle class with a Schubert class in
the torus-equivariant $K$-theory of a Kac-Moody flag manifold. The weight of the line bundle
may be dominant or antidominant, and the coefficients may be described either by
Lakshmibai-Seshadri paths or by the alcove model of the first author and Postnikov \cite{LP,LP1}.
For Lakshmibai-Seshadri paths, our formulas are the Kac-Moody generalizations 
of results of Griffeth and Ram \cite{GR} and Pittie and Ram \cite{PR} for finite dimensional flag manifolds.
A gap in the proofs of the mentioned results is addressed.
\end{abstract}

\maketitle

\section{Introduction}

\subsection{Chevalley formulas}
The Chevalley formula~\cite{Che} is a multiplication formula in the cohomology ring of 
a generalized flag variety $G/B$, where $G$ is a complex semisimple Lie group, and 
$B$ is a Borel subgroup. It expresses the action of multiplication by the class of a line bundle
$L^\la$ with respect to the basis of Schubert classes, and it implies a rule
for multiplication by a divisor class, known as Monk's rule in type~$A$.
The surjectivity of Borel's homomorphism in this setting implies that
the ring structure is determined by this rule.

One may consider the $K$-theory of $G/B$, equivariant with
respect to a maximal torus $T\subset B$,
and the multiplication by a line bundle class on the basis of 
equivariant classes of structure sheaves of Schubert varieties.
By forgetting equivariance and applying the Chern map, one may recover
the original Chevalley formula from an equivariant $K$-theoretic one.
The multiplication in this ring is again determined by a rule for
the product by line bundles: indeed, the equivariant Atiyah-Hirzebruch homomorphism
surjects.

The first such equivariant $K$-Chevalley formula, due to Pittie and Ram \cite{PR}, is described using
the combinatorics of Lakshmibai-Seshadri (LS) paths \cite{LS}. 
It applies to a dominant weight $\la$, and it is 
a positive formula (i. e., it has no cancellations). It was derived using a commutation relation 
in the $K$-theoretic affine nilHecke algebra of Kostant and Kumar \cite{KK:K}.
The same approach was applied to an antidominant weight 
$\la$ by Griffeth and Ram \cite{GR}, and the resulting formula is also cancellation-free.
The Pittie-Ram formula was explained geometrically  by Littelmann and Seshadri \cite{LS2}, 
based on standard monomial theory \cite{LS}. Willems \cite{Wi} gave a $K$-theory Chevalley formula 
for an arbitrary weight $\lambda$, which has cancellations even when $\lambda$ is dominant or antidominant.

A $K$-Chevalley formula which works for an arbitrary weight $\la$, and which is cancellation-free 
when $\la$ is dominant or antidominant, was given by the first author and Postnikov \cite{LP} 
in terms of their alcove model (cf. also \cite{LP1}). 
Besides allowing $\la$ to be an arbitrary weight, 
the alcove model formula has other advantages over the previous formulas. 
For example, the Deodhar lifts (from $W/W_\la$ to $W$, where $W_\la$ is the stabilizer of $\la$), 
which are required in the Pittie-Ram formula, and which are given by a 
nontrivial recursive procedure, are completely avoided. The only computations necessary
for the alcove model involve Bruhat covers and cocovers, together with a simple lex order.

We give four cancellation-free Chevalley formulas for the torus-equivariant $K$-theory of 
Kashiwara's Kac-Moody thick flag manifold \cite{Kas}: for dominant and antidominant weights,
both cases being described in terms of LS paths and the alcove model. Note that both  
LS paths and the alcove model were already available in the Kac-Moody setting.
Although it is known that the equivariant Borel and Atiyah-Hirzebruch maps need not be
surjective outside of the finite-dimensional case \cite{Ku}, strong evidence has arisen that
knowing the product by line bundles still determines the ring structure
in equivariant $K$-theory. For thick Kac-Moody flag schemes associated with root systems
of affine type, it is shown in \cite{KS} that, after inverting 
some scalars which multiply trivially against Schubert classes,
the equivariant Atiyah-Hirzebruch map is surjective.

\subsection{Approach to proofs} 
The LS path formulas are proved first, by extending the Pittie-Ram approach 
to the Kac-Moody setting. The first problem to address is the absence of a top Schubert class, 
which previously played a crucial role in the passage from the commutation relation in the 
Hecke algebra to the Chevalley formula. This is circumvented using
equivariant localization in $K$-theory and the action of the $K$-theoretic nilHecke ring 
on the torus equivariant $K$-theory of a Kac-Moody flag manifold. 
This method of computation was pioneered by
Kostant and Kumar \cite{KK:K} for Kac-Moody flag ind-varieties, but it is perhaps most naturally
interpreted in the geometry of the thick Kac-Moody flag schemes \cite{Kas,KS}.

Secondly, while working in the Kac-Moody setup, it is necessary to address a gap in the proofs in \cite{PR,GR}.\footnote{We 
correct the second formula in \cite[Theorem 3.5]{GR},
whose analogue is \eqref{comm2} in Theorem \ref{T:maintheorem}.} These proofs
are complete only when $\la$ is regular (i. e., when $W_\la$ and the lifts are trivial), 
as there is no treatment of the rather subtle interaction between the crystal graph structure 
on LS paths \cite{Li1,Li2} and the Deodhar lifts. The study of this interaction constitutes the technical
core of the proof of the LS path Chevalley formulas.

The alcove model Chevalley formulas are then derived from the LS path versions.
In \cite{LP1} a crystal graph isomorphism was given between LS paths and the alcove model;
it was defined by perturbing an LS path and taking a limit.
We refine this bijection to establish a bijection between the
subcollections of LS paths and alcove paths that are relevant to the Chevalley formulas.
The refined bijection depends on a new description of the 
bijection in \cite{LP1}. Some crucial ingredients in the refined bijection 
are the notion of EL-shellability of the Bruhat order on a Coxeter group based on 
Dyer's reflection orders \cite{Dyer},
and the description of LS paths using a weaker version
of the Bruhat order called the $b$-Bruhat order \cite{LS} (see also \cite{St}). These considerations completely clarify
the $\lambda$-chain or alcove model due to the first author and Postnikov \cite{LP,LP1}.

\subsection{Additional combinatorial consequences and examples}
As a byproduct, we obtain alcove path formulations for the Demazure and opposite Demazure subcrystals of highest
weight crystals for Kac-Moody algebras. For Demazure crystals,  
the alcove path formulation was previously known in \cite{LP} in the case of finite root systems.

Let us now point out some combinatorial features of LS paths and the alcove model
regarding Demazure and opposite Demazure crystals. 
LS paths are the paragon for realizing the crystal graphs of highest weight modules over
quantum groups. One of their historically well-known bonus features is 
the immediate detection of membership, via initial and final directions,
in a Demazure (resp. opposite Demazure) subcrystal; the latter is the crystal subgraph whose underlying module is a submodule
of a highest weight module generated from an extremal weight vector by the 
action of an upper (resp. lower) triangular subalgebra
$U_q(\mathfrak{b}_+)$ (resp. $U_q(\mathfrak{b}_-)$) of the quantum group $U_q(\geh)$. Another feature of LS paths is the efficient construction of Demazure crystals using stringwise crystal graph operations starting from the
highest weight vector \cite{Kas2}.

It is a consequence of the work on the alcove model \cite{LP, LP1} that, by transport through
the isomorphism between LS paths and alcove paths, the former can be equipped with some
operations which are fundamental to the alcove model, but which are new for LS paths. Indeed, 
to the authors' knowledge, the mentioned operations, based on Bruhat covers, were not previously part of the crystal graph technology
in the literature. 

In the alcove model, the most natural way to generate the Demazure crystal is to 
construct a rooted tree which starts at the \emph{lowest} weight vector. 
More precisely, the root vertex is labeled by any Weyl group element that
sends the highest weight of the Demazure module to its 
lowest weight, and a child of a tree vertex $v$ is determined by
a Bruhat cocover of $v$ together with an integer decoration, 
subject to a pruning condition that compares this edge with the one
coming down to $v$. The objects in the model are given by 
the paths from the root to any vertex.

For the alcove model, the opposite Demazure crystal is generated in an entirely similar
manner to the Demazure crystal. Instead, one
starts at the highest extremal weight vector,
and generates decorated Bruhat covers using a different but equally simple pruning condition for branches.
In contrast, the authors are unaware of a simple way to efficiently generate the opposite Demazure crystal using
crystal graph operations.

We conclude the paper with some examples. These include a 
Chevalley formula for the $K$-theory of the type $A$ affine Grassmannian $Gr_{SL_n}$, 
which uses the $n$-restricted partitions in the Misra-Miwa model \cite{mm}.

\subsection{Future work}
Since the alcove model $K$-Chevalley formula works in a uniform manner 
for an arbitrary weight $\la$ in the finite-dimensional setting,
we speculate that such a formula exists in Kac-Moody generality. We believe that 
the general approach in \cite{LP} can be extended in spite of some obvious obstacles. 
New ideas are required for the definitions
associated with a general weight $\la$, especially for weights outside the Tits cone, such as
weights of level zero for root systems of affine type.
Note that the LS path formula, used here as a starting point, is currently available 
only for dominant and antidominant $\la$. 

\subsection{Computer implementation}
We have implemented the various Kac-Moody Chevalley rules in \texttt{Sage} \cite{Sage},
an open source mathematics software system, using the \texttt{sage-combinat}
extension \cite{SC}. After fine-tuning and a period of peer review, these programs will be
incorporated into \texttt{sage-combinat} for free public distribution.

\subsection*{Acknowledgements}
We are grateful to Matthew Dyer for helpful discussions.
Thanks to Nicolas Thierry for technical assistance 
involving \texttt{sage-combinat} \cite{SC}, and to Anne Schilling for her work
on crystal graphs in \texttt{sage-combinat}.
Thanks to both Anne and Nicolas for a \texttt{sage} program which implements
localization of equivariant cohomology Schubert classes for flag manifolds
of finite and affine type, solving overnight a homework problem posed 
by the second author in a summer school lecture
at the Fields Institute in the summer of 2010.

\section{The Kac-Moody thick flag manifold and equivariant $K$-theory}

\subsection{Equivariant $K$-Chevalley coefficients}
Let $X$ be a {\em Kac-Moody thick flag manifold} \cite{Kas} over $\C$ with Dynkin
node set $I$. If the underlying Lie algebra $\geh$ is
infinite-dimensional then $X$ is a scheme of infinite type (as
opposed to the flag ind-scheme studied in \cite{KK:K,Ku}). It contains a canonical point
$x_0$, and it has an action of the Borel group $B$ and each of the
minimal parabolic subgroups $P_i$. Let $W$ be the Weyl group. For $w\in W$, let $\Xc_w := B
\cdot w x_0$ be the Schubert cell. $\Xc_w$ is the Spec of a
polynomial ring (with a countably infinite number of generators if
$\dim \geh = \infty$). Its Zariski closure is the {\em Schubert variety}
$X_w$. It has codimension $\ell(w)$ in $X$. There are cell decompositions
\begin{align}
  X = \bigsqcup_{w\in W} \Xc_w \qquad\qquad X_w = \bigsqcup_{\substack{v\in W \\ v \ge w}} \Xc_v
\end{align}
where $\ge$ is the Bruhat order on $W$. Let $S\subset W$ be a
nonempty finite Bruhat order ideal (if $w\in S$ and
$v\le w$ then $v\in S$). Let
\begin{align}
  \Omega_S := \bigcup_{w\in S} w \cdot \Xc_e = \bigsqcup_{w\in S}
  \Xc_w\,.
\end{align}
Let$K^B(\Omega_S)$ be the Grothendieck group of $B$-equivariant
coherent sheaves on the $B$-stable quasi-compact open subset $\Omega_S$. 
For $w\in W$ the structure sheaf $\OO_{X_w}$ is a $B$-equivariant
coherent $\OO_X$-module which by restriction defines a class
$[\OO_{X_w}]\in K^B(\Omega_S)$ provided that $w\in S$. Define
\begin{align}
  K^B(X) = \overset{\longleftarrow}{\lim_S}\, K^B(\Omega_S).
\end{align}

There is an isomorphism $K^T(X)\cong K^B(X)$, where $T\subset B$
is the maximal torus. We have \cite{KS}
\begin{align}\label{E:KTXSchubert}
  K^T(X) = \prod_{w\in W} K^T(\pnt) [\OO_{X_w}].
\end{align}
In particular, the product of two Schubert classes $[\OO_{X_u}]$
and $[\OO_{X_v}]$ in $K^T(X)$ may be an infinite $K^T(\pnt)$-linear combination
of classes $[\OO_{X_w}]$ if $X$ is infinite-dimensional.

For a weight $\la$, let $L^\la$ denote the $T$-equivariant line
bundle on $X$ of weight $\la$. Define the {\em equivariant $K$-Chevalley
coefficients} $a^w_{v\la}\in K^T(\pnt)$ by
\begin{align} \label{E:constprod}
  [L^\la] \,[\OO_{X_v}] = \sum_{w\in W} a^w_{v\la} [\OO_{X_w}].
\end{align}

Our main result, Corollary \ref{C:Chevalley}, gives explicit cancellation-free combinatorial formulas 
in terms of Lakshmibai-Seshadri (LS) paths for 
the Chevalley multiplicities $a^w_{v\la}$ if $\la$ is dominant or antidominant.
In Theorem \ref{kchev}, we express these Chevalley rules in terms of the $\la$-chains
of the first author and Postnikov \cite{LP,LP1}.

\subsection{The nilHecke ring and Chevalley coefficients}
The coefficients $a^w_{v\la}$ may be computed using the $K$-theoretic
\emph{nilHecke} ring of Kostant and Kumar \cite{KK:K}, which we recall here.

There are isomorphisms $K^T(\pnt) \cong R(T)\cong \Z[\La]$, where
$\La$ is the weight lattice. We shall use these identifications
without additional mention in the sequel. Let $\{\al_i^\vee\mid i\in I\}\subset\La^*=\Hom_{\Z}(\La,\Z)$
be the simple coroots, and $\ip{\cdot}{\cdot}:\La^*\times\La\to\Z$ be the evaluation pairing.
The Weyl group $W$ acts on $\La$ by $s_i \cdot \la = \la - \ip{\al_i^\vee}{\la}\al_i$
for $i\in I$ and $\la\in\La$. This induces an action of $W$ on $R(T)$ and $Q(T)=\Frac(R(T))$.
For $\la\in\La$, let $e^\la\in R(T)$ denote the isomorphism class of the one-dimensional
$T$-module of weight $\la$.

For $i\in I$, define the operator $T_i$ on the fraction field
$Q(T)=\Frac(R(T))$ by
\begin{align} \label{E:Tformula}
  T_i = (1 - e^{\al_i})^{-1}(s_i - 1).
\end{align}
For $\la\in\La$, we have \cite{LSS}
\begin{align} \label{E:TonE}
  T_i \cdot e^\la =\frac{e^\la-e^{s_i\la}}{e^{\alpha_i}-1}=
  \begin{cases}
  e^\la(e^{-\al_i}+e^{-2\al_i}\dotsm+e^{(-\ip{\al_i^\vee}{\la})\al_i}) &
  \text{if $\ip{\al_i^\vee}{\la} > 0$} \\
  0 & \text{if $\ip{\al_i^\vee}{\la} = 0$} \\
  - e^\la(1+e^{\al_i}+\dotsm+e^{(-1-\ip{\al_i^\vee}{\la})\al_i}) &
  \text{if $\ip{\al_i^\vee}{\la} < 0$.}
  \end{cases}
\end{align}
It follows that $T_i$ acts on $R(T)$. The {\em Demazure
operator} \cite{Dem} is $D_i=1+T_i$. The $T_i$ satisfy the braid relations and
$T_i^2=-T_i$ for all $i\in I$. Therefore we may define $T_w =
T_{i_1}\dotsm T_{i_N}$ for any reduced decomposition
$w=s_{i_1}\dotsm s_{i_N}$, where $i_1,\dotsc,i_N\in I$. Let $\K_0$ be
the $0$-Hecke algebra, which is the subring of $\End(R(T))$
generated by $\{T_i\mid i\in I\}$. We have the following identity of
operators on $R(T)$ \cite[(2.6)]{LSS}, where $e^\la$ denotes the
operator of left multiplication by $e^\la\in R(T)$:
\begin{align}\label{E:commute}
  T_i\, e^\la = (T_i \cdot e^\la) + e^{s_i\la} T_i.
\end{align}
The nilHecke ring $\K$ is by definition the smash product of $\K_0$
and $R(T)$ (acting on itself by left multiplication). By
\eqref{E:commute} we have
\begin{align}
  \K = \bigoplus_{w\in W} R(T) T_w.
\end{align}
Let $b^w_{v\la}\in K^T(\pnt)$ be defined by the following relation
in $\K$:
\begin{align}  \label{E:constcomm}
  T_w\, e^\la = \sum_{v\in W} b^w_{v\la}\, T_v.
\end{align}

\begin{lem}\label{L:consteq} We have
\begin{align}\label{E:consteq}
  a^w_{v\la} = b^w_{v\la}.
\end{align}
\end{lem}

\subsection{Localization and the proof of Lemma {\ref{L:consteq}}}
Lemma \ref{L:consteq} may be proved using localization. We use the
notation of \cite{LSS}. The computations here are essentially due to 
\cite{KK:K} but technically we are connecting the computations in
\cite{KK:K} to the thick Kac-Moody flag geometry.

Restriction to $T$-fixed points yields an injective ring homomorphism $\res: K^T(X) \to
K^T(X^T)$. Let $\Psi$ be the image of $\res$. Since $X^T \cong W$,
one may view $K^T(X^T)$ as the $K^T(\pnt)$-algebra $\F(W,R(T))$ of
functions $W\to K^T(\pnt)$.

Let $\psi^v = \res([\OO_{X_v}])$ for $v\in W$. Then
\eqref{E:KTXSchubert} translates to
\begin{align}
  \Psi = \prod_{v\in W} R(T) \psi^v.
\end{align}
The functions $\psi^v$ may be characterized as follows. Let $\K_Q :=
Q(T) \otimes_{R(T)} \K = Q(T) \otimes_{\Q} \Q[W]$, in light of
\eqref{E:Tformula}. Let $\Hom_{Q(T)}(\K_Q,Q(T))$ be the $\Q$-vector
space of left $Q(T)$-module homomorphisms $\K_Q\to Q(T)$. A function
$\psi\in \F(W,Q(T))$ may be regarded as an element of
$\Hom_{Q(T)}(\K_Q,Q(T))$ by
\begin{align}
\psi(\sum_{w\in W} a_w w) &= \sum_w a_w \psi(w)\,,
\end{align}
where $a=\sum_{w\in W} a_w w\in Q(T) \otimes_{\Q} \Q[W]$. Evaluation
defines a perfect pairing $$\ip{\cdot}{\cdot}: \K_Q \times
\Hom_{Q(T)}(\K_Q,Q(T)) \to Q(T).$$ This restricts to a perfect
pairing $\K \times \Psi \to R(T)$, and with respect to this pairing,
$\{T_w\mid w\in W\}$ and $\{\psi^v\mid v\in W\}$ are dual bases as
left modules over $R(T)$ \cite{LSS}:
\begin{align}\label{E:dualbases}
  \ip{T_w}{\psi^v} = \delta_{vw}\qquad\text{for $v,w\in W$.}
\end{align}

By abuse of notation, for $\la\in\La$, let $L^\la$ denote the image
under $\res$ of the class $[L^\la]\in K^T(X)$. We have
\begin{align}
  L^\la(w) = e^{w\la} \qquad\text{for all $w\in W$.}
\end{align}

\begin{lem}\label{L:linebundle} For any $\la\in\La$, $a\in \K$, and
$\psi\in \Psi$ we have
\begin{align} \label{E:linebundle}
  \ip{a}{L^\la \psi} = \psi(a e^\la).
\end{align}
\end{lem}
\begin{proof} Write $a\in \K$ as $a=\sum_{w\in W} a_w w$ for
$a_w\in Q(T)$. We have
\begin{align*}
  \ip{a}{L^\la \psi} &= (L^\la \psi)(a) \\
  &= (L^\la \psi)(\sum_w a_w w) \\
  &= \sum_w a_w (L^\la \psi)(w) \\
  &= \sum_w a_w L^\la(w) \psi(w) \\
  &= \sum_w a_w e^{w\la} \psi(w).
\end{align*}
On the other hand,
\begin{align*}
  \psi(a e^\la) &= \psi(\sum_w a_w w e^\la) \\
  &= \psi(\sum_w a_w e^{w\la} w) \\
  &= \sum_w a_w e^{w\la} \psi(w).
\end{align*}
\end{proof}

\begin{proof}[Proof of Lemma {\rm \ref{L:consteq}}] We have
\begin{align*}
  a^w_{v\la} &= \ip{T_w}{L^\la \psi^v} &\qquad&\text{by
  \eqref{E:constprod} and \eqref{E:dualbases}} \\
&= (L^\la \psi^v)(T_w) && \\
&= \psi^v(T_w e^\la) &\qquad&\text{by Lemma \ref{L:linebundle}} \\
&= \psi^v(\sum_u b^w_{u\la} T_u) &&\text{by \eqref{E:constcomm}} \\
&= \sum_u b^w_{u\la} \psi^v(T_u) && \\
&= \sum_u b^w_{u\la} \delta_{uv} && \text{by \eqref{E:dualbases}} \\
&= b^w_{v\la}.
\end{align*}
\end{proof}

\subsection{Recurrence}
A recurrence for the coefficients $b^w_{z\la}$ is obtained as follows. Assuming that $s_iw<w$, we have
\begin{align*}
  \sum_{z\in W} b^w_{z\la} T_z  &=
  T_w e_\la \\
  &= T_i T_{s_i w}  e^\la \\
  &= T_i \sum_y b^{s_iw}_{y\la} \,T_y \\
  &= \sum_y \left( T_i \cdot b^{s_iw}_{y\la} + s_i(b^{s_iw}_{y\la}) T_i\right) T_y\,. 
\end{align*}
Using
\begin{align}
  T_i T_y = \begin{cases}
  T_{s_iy} & \text{if $s_iy>y$} \\
  - T_y & \text{if $s_iy<y$}
  \end{cases}
\end{align}
and taking coefficients of $T_z$ on both sides, we have
\begin{align}
  b^w_{z\la} = T_i \cdot b^{s_iw}_{z\la} +
  \chi(s_iz<z) \,s_i(b^{s_iw}_{s_iz,\la} -   b^{s_iw}_{z\la}).
\end{align}
Summarizing, for $s_iw<w$ and $z<s_iz$ we have
\begin{align}
\label{E:bz}
  &b^w_{z\la} = T_i \cdot b^{s_iw}_{z\la}\\
\label{E:bsz}
  &b^w_{s_iz,\la} = T_i \cdot b^{s_iw}_{s_iz,\la} + s_i(b^{s_iw}_{z\la}-b^{s_iw}_{s_iz,\la})\,.
\end{align}
By iterating this recurrence, one may obtain an explicit formula for
$b^w_{v\la}$ which typically has a lot of cancellation.
Fix a reduced word $a_1a_2\dotsm a_N$ of $w$. Let $E(v,a)$ be the
set of sequences $\epsilon=(\epsilon_1,\dotsc,\epsilon_N)\in
[0,1]^N$ such that
\begin{align*}
  \prod_{\substack{i\\ \epsilon_i=1}} T_{a_i} = \sgn(\epsilon) T_v\,,
\end{align*}
where $\sgn(\epsilon)\in \{\pm 1\}$, and the product is ordered from
left to right by increasing $i$. Here note that $T_i^2=-T_i$. For
$\epsilon\in E(v,a)$, let $A_\epsilon=A_1A_2\dotsm A_N$, where
\begin{align*}
  A_i = \begin{cases}
  s_{a_i} & \text{if $\epsilon_i=1$} \\
  T_{a_i} & \text{if $\epsilon_i=0$.}
  \end{cases}
\end{align*}
Then we have
\begin{align} \label{E:explicitformula}
  b^w_{v\la} = \left(\sum_{\epsilon\in E(v,a)} \sgn(\epsilon) A_\epsilon \right) \cdot e^\la.
\end{align}

\begin{ex} In type $A_2$ consider $w=s_1s_2s_1$, $v=s_1$, and
$\la=2\omega_1+\omega_2=(310)$. Using the above reduced word $a$ for
$w$, we have $E(v,a)=\{(100),(001),(101)\}$ and
\begin{align*}
  b^w_{v\la} &= (s_1T_2T_1+T_1T_2s_1-s_1T_2s_1)\cdot e^{310} \\
&=e^{022}+e^{013}.
\end{align*}
\end{ex}

\section{The Chevalley formula in terms of LS paths}

\subsection{The main result}\label{mainresls} We recall the definition of \emph{Deodhar lifts} 
in Coxeter groups. For more details the reader may see Sections \ref{sbruhat} and \ref{deolifts}.

Let $W$ be a Coxeter group and let $S=\{s_i\mid i\in I\}$ be the set of simple reflections.
A reflection in $W$ is a $W$-conjugate of a simple reflection. The length $\ell=\ell(w)$
of $w\in W$ is the minimum $\ell$ such that $w = s_{i_1} s_{i_2}\dotsm s_{i_\ell}$ 
for $i_1,\dotsc,i_\ell\in I$. The (strong) Bruhat order $\le$ on $W$ is the partial order with covering relation
$v\lessdot w$ if $\ell(w)=\ell(v)+1$ and there is a reflection $r\in W$ such that $w = v r$.

For a subset $J\subset I$, let $W_J$ denote the subgroup of $W$ generated by $s_i$ for $i\in J$.
Let $W^J$ be the set of minimum length coset representatives in $W/W_J$.
The set $W^J$ inherits the Bruhat order from $W$.
The Bruhat order $\le$ on $W/W_J$ is defined by declaring that the bijection
$W^J\to W/W_J$ given by $w\mapsto w W_J$ is an isomorphism of posets.

\begin{prop} \label{P:Deodhar} \cite{D}  {\rm (1)} Let $\tau\in W/W_J$ and $v\in W$ be such that $v W_J \le \tau$ in $W/W_J$. Then the set 
\[  \{  w\in W\mid \text{$v\le w$ and $wW_J=\tau$} \}\]
has a Bruhat-minimum, which will be denoted by $\up(v,\tau)$.

{\rm (2)} Let $\tau\in W/W_J$ and $w\in W$ be such that $w W_J\ge\tau$ in $W/W_J$. Then the set 
\[  \{  v\in W\mid \text{$w\ge v$ and $vW_J=\tau$} \}\]
has a Bruhat-maximum, which will be denoted by $\down(w,\tau)$.\footnote{$\down$ is an abbreviation for ``down''.}
\end{prop}

Fix a dominant weight $\la\in\La^+$. Let $\Ta^\la$ be the set of \emph{Lakshmibai-Seshadri (LS) paths} 
of shape $\la$ \cite{LS}. For a precise characterization of LS paths, see Section \ref{ls-nonrec}. Recall that an LS path $p$ is a piecewise-linear map 
$p\::\:[0,1]\rightarrow {\mathfrak h}_{\mathbb R}^*$ with $p(0)=0$.
The endpoint of the path is $p(1)$. A path $p\in \Ta^\la$ may also be defined by a sequence of vectors
of the form $a_1 v_1, a_2 v_2,\dotsc, a_m v_m$, where the $a_i$ are positive rational numbers
summing to $1$, and the vectors $v_i$ are in the orbit $W\cdot \la$. The path is given by the sequence of points
$0$, $a_1v_1$, $a_1v_1+a_2v_2$, $\dotsc$, $a_1v_1+\dotsm+a_mv_m$.
The stabilizer $W_\la$ of $\la$
equals $W_J$, where $J = \{i\in I\mid s_i\cdot\la = \la\}$. If a vector is a positive real multiple
of an element of $W\cdot\la$, then we say that its direction is the corresponding element of $W/W_\la$.
This given, the directions of the vectors in an LS path decrease in the Bruhat order on $W/W_\la$. Denote
by $\iota(p) \in W/W_\la$ (resp. $\phi(p)\in W/W_\la$) the initial (resp. final) direction of $p$, i.e., 
the direction of the first (resp. last) vector in $p$.

Given an LS path $p\in \Ta^\la$ and $z\in W$ such that $z W_\la \le \phi(p)$, let $\up(w,p)\in W$ be defined
as follows. Let the LS path $p$ have directions 
\[\iota(p)=\sigma_1\ge\ldots\ge\sigma_m=\phi(p)\,.\]
Define the sequence of Weyl group elements
\[z=w_{m+1}\le w_m\le\ldots\le w_1=\up(z,p)\]
recursively by $w_i:=\up(w_{i+1},\sigma_i)$ for $i$ from $m$ down to $1$.
Given $w\in W$ such that $\iota(p) \le w W_\la$, define
\[w=w_0\ge w_1\ge\ldots\ge w_m=\down(w,p)\] 
by $w_i:=\down(w_{i-1},\sigma_i)$ for $i$ from $1$ to $m$.

For $z,w\in W$, define
\begin{align}
\label{E:uppaths}
  \Pu^\la_{w,z} &= \{\, p\in \Ta^\la \mid \text{$\phi(p)\ge zW_\la$ and
  $\up(z, p) = w$}\,\}\,,\\
\label{E:downpaths}
\Pd^\la_{w,z} &= \{\, p\in \Ta^\la \mid \text{$\iota(p)\le wW_\la$ and $\down(w,p)=z$}\,\}\,.
\end{align}

\begin{thm}\label{T:maintheorem} We have
\begin{align}
\label{comm2}
T_{w}\,e^{\la}
 &= \sum_{z\le w} \sum_{p\in\Pu^\la_{w,z}} e^{p(1)} \, T_z\,,\\
\label{comm1}
T_w \,e^{-\la} &= \sum_{\substack{p\in\Ta^\la \\ \iota(p)\le wW_\la}} 
(-1)^{\ell(w)-\ell(\down(w,p))} e^{-p(1)}\, T_{\down(w,p)}\,.
\end{align}
\end{thm}

\begin{rem}\label{R:grerr} Formula \eqref{comm2} corrects the second formula in
\cite[Theorem 3.5]{GR}. To make the translation to compare the formulas, one must use the
algebra involution on the $K$-theoretic nilHecke ring given by 
$e^\la\mapsto e^{-\la}$ (as left multiplication operators) and $T_i\mapsto -1-T_i$ (or, equivalently, $D_i\mapsto 1-D_i$). 
\end{rem}

By Lemma \ref{L:consteq} we obtain the following result.

\begin{cor} \label{C:Chevalley} We have
\begin{align}
\label{chdom}
[L^\la]\,[\OO_z] &= \sum_{\substack{p\in \Ta^\la \\ \phi(p) \ge zW_\la}} e^{p(1)} [\OO_{\up(z,p)}]\,,\\
\label{chantidom}
[L^{-\la}]\,[\OO_z] &= \sum_{w\ge z} \sum_{p\in\Pd^\la_{w,z}} (-1)^{\ell(w)-\ell(z)}\, e^{-p(1)}\, [\OO_w]\,.
\end{align}
\end{cor}

\begin{rem}\label{R:antidominant} We have described the antidominant formula \eqref{comm1} using the
combinatorics of the highest weight crystal $\Ta^\la$ with $\la$ dominant. Conceptually it is
better to use the lowest weight crystal $\Ta^{-\la}$; for infinite-dimensional Lie algebras
these are not highest weight crystals, as they are in the finite-dimensional setting.
The natural parametrization of the directions in an LS path in $\Ta^{-\la}$,
gives a sequence that is increasing in the Bruhat order. 
We find it convenient to work with the more familiar 
objects in $\Ta^\la$. This is achieved by applying the crystal antiautomorphism $\Ta^{-\la}\to\Ta^\la$,
which is defined by reversing the sequence of vectors and negating them, that is,
sending the sequence of vectors $a_1v_1,\dotsc,a_mv_m$ to $-a_mv_m,\dotsc,-a_1v_1$.
This contragredient duality map has the effect of negating weights and
reversing colored arrows. We do not see a way to
deduce the antidominant formula from the dominant one.
\end{rem}

Proposition \ref{P:pathlift} below is the crucial result needed to prove relation \eqref{comm2}. To state it,
recall that $\Ta^\la$ is a model for the {\em crystal graph} of the irreducible $U_q(\geh)$-module 
of highest weight $\lambda$ where $U_q(\geh)$ is the quantum group. As such, $\Ta^\la$ is a graph
with vertex set $\Ta^\la$ and directed edges colored by 
the Dynkin node set $I$. The connected components of the restriction of the graph to arrows labeled $i$ for a fixed $i$,
are finite directed paths called \emph{$i$-strings}. 
For more details on the crystal graph structure on LS paths we refer to Section \ref{lslifts}.

Fix an $i$-string $\St$ in $\Ta^\la$. Let $h$ be the head or source of the string $\St$, $t$ its tail,
and $m=\St\setminus\{h,t\}$ its middle. Provisionally, for all $w,z\in W$, we write
$\Pu_{w,z}(\St)=\Pu^\la_{w,z} \cap \St$ and $s=s_i$. 
\begin{prop}\label{P:pathlift} Let $z\in W$ be such that $s z > z$ and $zW_\la \le \phi(h)$ and write $w=\up(z,h)$. 
Then $s w>w$ and
\begin{equation}\label{lift0}
\text{$\{\,\up(z,p) \mid p\in\St\,\}\subseteq\{w,sw\}$ and $\phi(p) \ge zW_\la$ for $p\in\St$}
\end{equation}
In terms of $\up(sz,\,\cdot\,)$ on $\St$, we have the following three disjoint cases:
\begin{align}
\label{lift1}
&\text{$\phi(p) \ge szW_\la$ for $p\in\St$ and $\{\up(sz,p) \mid p\in\St\,\}\subseteq\{w,sw\}$} \\
\label{lift2}
&\text{$\phi(t) \ge szW_\la$, $\up(sz,t)\in\{w,sw\}$, and $\phi(p) \not\ge szW_\la$ for $p\in\St\setminus\{t\}$} \\
\label{lift3}
&\text{$\phi(p) \ge szW_\la$ for $p\in\St$, $\up(sz,t)\in\{w,sw\}$, and $\{\up(sz,p) \mid
p\in\St\setminus\{t\}\,\}\subseteq\{\tw,s\tw\}$} 
\end{align}
where $\tw\notin \{w,sw\}$ and $s\tw >\tw$. More precisely, the following chart
gives the pairwise disjoint possibilities in terms of arbitrary $x$ with $sx>x$.
\begin{align}\label{E:chart}
\begin{array}{|c||c|c|c||c|c|c|||c|c||} \hline
\Pu_{x,z}(\St) &    \multicolumn{3}{c||}{\St} & \multicolumn{3}{|c|||}{ h } &\multicolumn{2}{|c||}{ \emptyset } \\ \hline 
\Pu_{s x,z}(\St) &    \multicolumn{3}{c||}{\emptyset} & \multicolumn{3}{|c|||}{ \St\setminus h } 
&\multicolumn{2}{|c||}{ \emptyset } \\ \hline\hline
\Pu_{x,sz}(\St) & t & \St & \{h,t\}& \emptyset  & h& \St\setminus t &\St\setminus t&h \\ \hline 
\Pu_{sx,sz}(\St)&\emptyset &\emptyset & m & t & \St\setminus h & t& \emptyset& m  \\ \hline\hline
|\St| & {\ge 1} & {\ge 2} & {\ge 3} & { \ge 1 }& {\ge 2}& {\ge 3} & {\ge 2 }&{\ge3} \\ \hline \hline
 \mathrm{Cases} & \UU.1.1 &\UU.1.2&\UU.1.3 & \UU.2.1 &\UU.2.2 & \UU.2.3 & {\rm \UU.3.1} & {\rm \UU.3.2} \\ \hline
\end{array}
\end{align}
Cases {\rm U}.a.b for $a\in\{1,2\}$ and $b\in\{2,3\}$ correspond to {\rm \eqref{lift1}} when $x=w$. 
Cases {\rm U.1.1} and {\rm U.2.1} correspond to \eqref{lift2} and \eqref{lift3} when $x=w$. 
Cases {\rm U.3.*} correspond to \eqref{lift3} when $x=\tw$.
\end{prop}

Given Proposition \ref{P:pathlift} one may deduce the dominant weight LS-path Chevalley rule \eqref{comm2}.

\begin{proof}[Proof of relation \eqref{comm2}]
Given a set of paths ${\mathcal P}$, let
\[\Sigma({\mathcal P}):=\sum_{p\in{\mathcal P}} e^{p(1)}\,.\]
Let $x, z\in W$ be such that $sx>x$ and $sz>z$.
It suffices to show that the coefficients in \eqref{comm2} satisfy the recurrence relations 
\eqref{E:bz} and \eqref{E:bsz} for the Chevalley coefficients, namely,
\begin{align*}
\Sigma(\Pu_{sx,z})=T_i\cdot \Sigma(\Pu_{x,z})\qquad\Sigma(\Pu_{sx,sz})=
T_i\cdot \Sigma(\Pu_{x,sz})+s\left(\Sigma(\Pu_{x,z})-\Sigma(\Pu_{x,sz})\right).
\end{align*}
Since $\Ta^\la$ is partitioned into $i$-strings, it suffices to establish these relations
for every $i$-string $\St$ with $\Pu_{u,v}$ replaced by $\Pu_{u,v}(\St)$ for all $u,v$.
In every column of the above table, it is easy to verify these relations. For instance, in Case U.1.1, 
the first relation $T_i\cdot\Sigma(\St)=0$ follows by acting with $T_i$ on
\[\Sigma(\St)=T_i\cdot e^{h(1)}+e^{h(1)}\,,\]
which is \eqref{E:TonE}. The second relation, in the same case, amounts to
\[T_i\cdot e^{s\,h(1)}+\Sigma(\St)-e^{h(1)}=0\,,\]
which is also \eqref{E:TonE}. All the other relations follow in a similar way.
\end{proof}

The proof of Proposition \ref{P:pathlift} occupies the next several subsections.

\begin{rem} If $\lambda$ is regular (that is, $W_\lambda=\{1\}$) then the Deodhar lifts are trivial. For such $\lambda$,
by Proposition \ref{P:infin} only cases U.$a.b$ for $a,b\in\{1,2\}$ can arise.
\end{rem}

\subsection{The Bruhat order}\label{sbruhat}

We recall some basic properties of the Bruhat order on a Coxeter group, beginning with the well-known {\em Z-property}.

\begin{prop}\cite{Hum} \label{P:Bruhat} 
Let $sv<v$ and $sw<w$. Then $v\le w$ if and only if $sv \le w$ if and only if $sv\le sw$. 
\end{prop}

\begin{prop} \label{P:coset} \cite{D} {\rm (1)} We have $W^J=\{w\in W\mid \text{$ws > w$ for all $s\in S\cap W_J$} \}\,.$\\
\indent {\rm (2)} Given $w\in W$ there exist unique $w^J\in W^J$ and $w_J\in W_J$ such that
$w=w^Jw_J$. Moreover $\ell(w)=\ell(w^J)+\ell(w_J)$. \\
\indent {\rm (3)} Given $w\in W^J$ and $s\in S$, exactly one of the following occurs:\\
\indent\indent {\rm (a)} $sw < w$. Then $sw\in W^J$ and $swW_J<wW_J$.\\
\indent\indent {\rm (b)} $sw > w$ and $sw \in W^J$. Then $swW_J>wW_J$.\\
\indent\indent {\rm (c)} $sw > w$ and $sw\not\in W^J$. Then $sw=ws'$ for some $s'\in S\cap W_J$
and $swW_J=wW_J$.\\
\indent {\rm (4)} For $v^J,w^J\in W^J$ and $v_J,w_J\in W_J$,
if $v=v^Jv_J \le w=w^Jw_J$ then $v^J\le w^J$ and $vW_J\le wW_J$.
\end{prop}

\begin{lem} \label{L:Bruhatcoset} Consider $\sigma, \,\tau$ in $W/W_J$ and $s\in S$. 
Suppose $s\sigma\le \sigma$ and $s\tau\le\tau$. 
Then $\sigma\le\tau$ if and only if $s\sigma\le\tau$ if and only if $s\sigma\le s\tau$. 
\end{lem}

\begin{proof}
Let $v,w\in W^J$ be such that $vW_J=\sigma$ and $wW_J=\tau$. By Proposition \ref{P:coset} (3), 
if $s\sigma<\sigma$ then $s v<v$ and $sv\in W^J$, while if $s\tau<\tau$ then $sw<w$ and $sw\in W^J$. 
If both of these cases hold, the result is immediate by Proposition \ref{P:Bruhat}. 
If $s\sigma=\sigma$ and $s\tau<\tau$, then $sv>v$, by Proposition \ref{P:coset} (3). 
It suffices to show that $\sigma\le \tau$ implies $\sigma\le s\tau$, i.e., $v\le w$ implies $v\le sw$; 
but this follows from Proposition \ref{P:Bruhat}. Finally, assume that $s\sigma<\sigma$ and $s\tau=\tau$. 
Like above, by Proposition \ref{P:coset} (3), we have $w<sw=ws'$ for $s'\in W_J$. It now suffices to 
show that $s\sigma\le\tau$ implies $\sigma\le\tau$; but this follows since $sv\le w$ implies $v\le sw=ws'$, 
by Proposition \ref{P:Bruhat}, which in turn implies $v\le w$, by Proposition \ref{P:coset} (4). 
\end{proof}

\subsection{Deodhar lifts}\label{deolifts}

The setup is still that of Coxeter groups. Recall from Proposition \ref{P:Deodhar} (2) the definition of the Deodhar lifts $\mn(v,\tau)$.

\begin{lem} \label{L:lowlift} Let $v\in W$ and $\tau\in W/W_J$ be such that $v W_J \le \tau$.
Let $s\in S$. If $s\tau>\tau$ (resp. $s\tau<\tau$), then $sw>w$ (resp. $sw<w$). If $\tau=s\tau$ and $sv>v$, then $sw>w$.
\end{lem}

\begin{proof} Since $wW_J=\tau$, the first statement is clear by Proposition \ref{P:coset} (4). Now let us consider the case $\tau=s\tau$.
Suppose $sw<w$. Since $v\le w$ and $sv>v$, we have $v\le sw$, by Proposition \ref{P:Bruhat}. But $swW_J=s\tau=\tau$. By the minimality of $w$ we obtain the contradiction $w \le sw$.
Therefore $sw>w$, as required.
\end{proof}

\begin{lem} \label{L:liftleft} Suppose $v\in W$ and $\tau\in W/W_J$
are such that $sv>v$ and $vW_J \le s\tau<\tau$.
Letting $y=\mn(v,s\tau)$ and $w=\mn(v,\tau)$, we have 
$y=sw<w$.
\end{lem}

\begin{proof} By Lemma \ref{L:lowlift}, we have $sy>y$ and $w>sw$. 
Since $sy>y\ge v$ and $syW_J=wW_J$, by the minimality of $w$ we have $w\le sy$.
Then Proposition \ref{P:Bruhat} gives $sw\le y$.

On another hand, since $v\le w$, we have $v\le sw$, by Proposition \ref{P:Bruhat}. 
Since $swW_J=yW_J$, by the minimality of $y$ we have $y\le sw$. Therefore $y=sw$.
\end{proof}


\begin{lem} \label{L:liftright} Let $v\in W$, $s\in S$, and $\tau\in W/W_J$ be such that $v<sv$, $sv W_J \le \tau$,
and $s\tau\le \tau$. Let $y=\mn(sv, \tau)$ and $w=\mn(v,\tau)$. Then $y=w$ or $y=sw>w$, and the latter holds only if $s\tau=\tau$.
\end{lem}

\begin{proof} Since $v<sv\le y$ and $y W_J=w W_J$,
by the minimality of $w$ we have $w\le y$.
Suppose first that $s\tau<\tau$. By Lemma \ref{L:lowlift}, we have $sw<w$.
Since $v\le w$, we have $sv\le w$, by Proposition \ref{P:Bruhat}. The minimality of $y$ implies $y\le w$, so $y=w$.
Now consider the case $s\tau=\tau$. By Lemma \ref{L:lowlift}, we have $sw>w$.
Proposition \ref{P:Bruhat} gives $sv\le sw$.
We have $swW_J=s\tau=\tau$, so the minimality of $y$ implies $y\le sw$.
Thus we proved $w\le y\le sw$, as required.
\end{proof}

\begin{ex} If $v<sv$ but $\tau<s\tau$, then $y$ can be something completely different.
Take type $A_3$, $J=\{1\}$, $v=s_3$, $s=s_1$, and $\tau=s_2s_3 W_J$. Then $w=s_2s_3$
and $y=s_2s_3s_1$. 
\end{ex}

\subsection{LS paths and lifts}\label{lslifts} We revert to the setup of root systems for symmetrizable Kac-Moody algebras. 
For a detailed description of LS paths we refer to \cite{LS,St}. The set $\Ta^\la$ of LS paths of shape $\lambda$ can be characterized as the set generated by {crystal operators} $f_i$, starting from the straight-line path from $0$ to $\la$
\cite{Li2}. The non-recursive description of LS paths is given in Section \ref{ls-nonrec}. 

Let $i\in I$ be a fixed Dynkin node. Recall the decomposition of the crystal graph $\Ta^\la$ into $i$-strings.
The crystal operator $f_i$ (resp. $e_i$) on $\Ta^\la$ is the partial operator on $\Ta^\la$
(function mapping from a subset of $\Ta^\la$ into $\Ta^\la$) that sends a vertex $p\in \Ta^\la$
to the next (resp. previous) vertex on its $i$-string if that vertex exists, and is otherwise undefined
on $p$. We now recall the explicit definitions of $f_i$ and $e_i$ on $\Ta^\la$.

The element $p\in \Ta^\la$ is now viewed as a sequence of vectors.
Every vector $v$ in $p$, has $i$-height $\ip{\alpha_i^\vee}{v}$ given by an integer
$n$. The vector $v$, whose direction has the form $w\cdot \la$ for some 
$w\in W$, is cut into $|n|$ copies of the vector $\frac{1}{|n|} v$.
All vectors in $p$ are still pointing in some direction in $W\cdot\la$,
but now each has $i$-height in the set $\{-1,0,1\}$.
In the sequel, when discussing the crystal graph structure on $\Ta^\la$,
we will abuse notation by replacing each such step $v$ by its direction, i.e., the 
corresponding element of $W/W_\la \cong W\cdot \la$. The following rule (called the signature rule)
defines the actions of $f_i$ and $e_i$. Consider the sequence of steps
(indexed backwards)
$\phi(p)=\sigma_1 \le \sigma_2 \dotsm \le \sigma_\ell=\iota(p)$ in $p$.
The \emph{$i$-signature} of $p$ is by definition the word
$h_1h_2\dotsm h_\ell$, where $h_k\in \{-1,0,1\}$ is the $i$-height
of $\sigma_k$ for all $k$. We view each $-1$ as a left parenthesis, and each $+1$
as a right parenthesis; the $0$s are ignored. Pairing parentheses as usual,
the unpaired subsequence consists of some number of $+1$'s followed by some number of $-1$'s.
The \emph{signature rule} declares that $f_i(p)$ (resp. $e_i(p)$) is obtained from $p$
by taking the rightmost unpaired $+1$ (resp. leftmost unpaired $-1$), say $h_j$,
and replacing $\sigma_j$ by $s_i \sigma_j$; we have $s_i \sigma_j>\sigma_j$ (resp. $s_i\sigma_j<\sigma_j$).
If the mentioned $+1$ (resp. $-1$) does not exist, then $f_i(p)$ (resp. $e_i(p)$) is undefined.
Passing from $p$ to $f_i(p)$ (resp. $e_i(p)$) affects the $i$-signature by
changing the $j$-th symbol from $+1$ to $-1$ (resp. $-1$ to $+1$) and leaving other symbols unchanged.

Let $\St\subset \Ta^\la$ be a fixed $i$-string. It may be written
\begin{align}
\label{E:thestring}
\St=\{h=p_0,\,p_1,\,\ldots,\, p_m=t\}.
\end{align} where $h$ is the head, $t$ is the tail, and $p_k=f_i^k(h)$. Let 
\begin{equation}
\label{E:iotaphi} 
\phi_k:=\phi(p_k)\,,\;\;\;\;\;\;\;\;
\iota_k:=\iota(p_k)\,,
\end{equation}
for $0\le k\le m$. The following standard result is a consequence of the signature rule. 

\begin{prop}\label{P:infin} If $|\St|\ge 2$, we have 
\begin{align}
\label{incases}
\iota_0&=\iota_1=\ldots=\iota_m<s_i\iota_0&\quad&\text{or}&\quad \iota_0&<\iota_1=\ldots=\iota_m=s_i\iota_0 \\
\label{fincases}
s_i\phi_m&<\phi_0=\ldots=\phi_{m-1}=\phi_m&\quad&\text{or}&\quad s_i\phi_m&=\phi_0=\ldots=\phi_{m-1}<\phi_m\,.
\end{align}
If $|\St|=1$, we have $s_i\iota_0\ge \iota_0$ and $s_i\phi_0\le\phi_0$. 
\end{prop}

We now study the way the lifts of LS paths change along the string $\St$ in \eqref{E:thestring}.

\begin{lem}\label{L:headlift} Let $z\in W$ be such that $zW_\la \le \phi(h)$, and set $u = \up(z,h)$. 
Assume that $|\St|\ge 2$ or $z<s_i z$. Then we have $u<s_i u$.
\end{lem}

\begin{proof} 
Let $h=(\sigma_1\le\sigma_2\le\dotsm\le\sigma_\ell)$.
Let $j$ be largest such that $\sigma_j\ne s_i\sigma_j$, assuming that such an index exists. 
In particular, this happens when $|\St|\ge 2$. Since $h$ is the head, by the signature rule 
we have $\sigma_j<s_i \sigma_j$.  Let $w_0=z$ and $w_k=\mn(w_{k-1},\sigma_k)$ for $1\le k\le \ell$.
By applying Lemma \ref{L:lowlift} repeatedly, we deduce first $w_j<s_i w_j$, and then $u<s_i u$. 

We are left with the case when $\sigma_k=s_i\sigma_k$ for all $k$. But then $|\St|=1$, 
so we can use the assumption $z=w_0<s_i w_0$. We conclude the proof as above, by applying 
Lemma \ref{L:lowlift} repeatedly.
\end{proof}

\begin{lem} \label{L:edgelift} Let $p,p'\in \St$ with $f_i(p)=p'$. 
Let $z\in W$ be such that $z W_\la \le \phi(p)$, so we can define $u:=\up(z,p)$ and $u':=\up(z,p')$. 
Suppose that $p'\ne t$ or $z<s_i z$.
Then $u'=u$ or $u'=s_i u>u$, and the latter occurs only if $p=h$.
\end{lem}

\begin{proof} 
Denote the steps of $p$ (indexed in reverse order) by
\begin{align*}
\phi(p)=\sigma_1 \le \sigma_2 \le\dotsm\le\sigma_\ell = \iota(p)\,.
\end{align*} 
We know that $p'$ is obtained from $p$ by replacing a step
$\sigma_j$ by $s_i \sigma_j>\sigma_j$. 
Setting $z=w_0=w'_0$, let
$w_k = \mn(w_{k-1},\sigma_k)$ for $1\le k\le \ell$, and define $w'_k$ similarly for
$p'$ instead of $p$. Then $u=w_\ell$ and $u'=w_\ell'$. 

Let $1\le r\le j$ be smallest such that $s_i \sigma_k = \sigma_k$ for all $r\le k <j$.
If $r>1$ then the signature rule implies $\sigma_{r-1}<s_i\sigma_{r-1}$. By applying Lemma \ref{L:lowlift} 
repeatedly, we deduce first
\begin{equation}\label{wr1}
w_{r-1}<s_i w_{r-1}\,,
\end{equation}
and then $w_{j-1} < s_i w_{j-1}$.  Alternatively, if $r=1$, then $p'=t$ by the signature rule, 
so we must have $z<s_i z$. But this means that \eqref{wr1} again holds, so we can deduce 
$w_{j-1}<s_i w_{j-1}$ like above. On another hand, we clearly have $w'_k=w_k$ for $0 \le k < j$.
Lemma \ref{L:liftleft} then gives $w'_j=s_i w_j>w_j$.

Let $j\le q\le \ell$ be largest such that $s_i \sigma_k=\sigma_k$ for $j < k \le q$.
By repeated applications of Lemma \ref{L:liftright}, we deduce that either $w'_q=w_q$,
so that $w'_\ell=w_\ell$ and we are done, or (as we shall assume) $w'_q=s_iw_q>w_q$.
If $q=\ell$ then $w'_\ell = s_i w_\ell>w_\ell$; moreover, as there were no steps
$\sigma_k$ with negative $i$-height for $k>j$, we have $p=h$.
Otherwise, we have $q<\ell$ and, by the signature rule, $\sigma_{q+1}>s_i \sigma_{q+1}$.
By Lemma \ref{L:liftright}, we have $w'_{q+1}=w_{q+1}$, and therefore $w'_\ell=w_\ell$.
\end{proof}

\begin{lem}\label{L:endstring} Consider $z\in W$ satisfying $z<s_i z$ and
$s_i zW_\la \le \phi(t)$. Let $u=\up(z,t)$ and $u'=\up(s_i z,t)$. 
If $|\St|\ge 2$, we always have $u'=u$. If $|\St|=1$, then we have $u'=u<s_i u$ or $u'=s_i u>u$. 
\end{lem}

\begin{proof} Let $t=(\sigma_1\le\sigma_2\le\dotsm\le\sigma_\ell)$,
$w_0=z$, $w'_0=s_i z$, $w_k = \mn(w_{k-1},\sigma_k)$, 
and $w'_k=\mn(w'_{k-1},\sigma_k)$, for $1\le k\le \ell$. We have $u=w_\ell$ and $u'=w_\ell'$.

Let $0\le j\le\ell$ be largest such that $\sigma_k = s_i \sigma_k$ for $1\le k\le j$.  
By repeated applications of Lemma \ref{L:liftright}, we have $w'_j=w_j$, which implies $u'=u$, 
or $w'_j = s_i w_j>w_j$. If $j=\ell$, then the relationship between $u'$ and $u$ is established. 
Otherwise, in the latter case, we have $\sigma_{j+1} > s_i \sigma_{j+1}$, by the signature rule.
Lemma \ref{L:liftright} then gives $w'_{j+1}=w_{j+1}$, and therefore $u'=u$ once again.

Note that $j=\ell$ only if $|\St|=1$, so $|\St|\ge 2$ implies $j<\ell$. In the case $|\St|=1$, 
we also need to show that $u<s_i u$ in general (this is already known if $u'=s_i u$). 
This follows from Lemma \ref{L:headlift}, since now $t=h$.
\end{proof}

We called $\St\setminus\{h,t\}$ the middle of the $i$-string $\St$. 

\begin{lem} \label{L:Pmiddle} If $\Pu^\la_{w,z}\cap \St$ contains any element in the middle of $\St$,
then it contains the entire middle of $\St$.
\end{lem}

\begin{proof} This follows immediately from Proposition \ref{P:infin} and Lemma \ref{L:edgelift}.
\end{proof}


\subsection{The proof of Proposition \ref{P:pathlift}} We retain the notation from the
previous subsection, in particular \eqref{E:thestring} and
\eqref{E:iotaphi}. Let $z\in W$ be such that
\begin{align}\label{E:zdom}
  z<s_i z\,.
\end{align}
For the rest of this section we assume that at least one of the lifts $\up(z,p)$ or $\up(s_i z,p)$ 
are defined for some $p$ in $\St$. That is, at least one of the following statements holds for some $k$:
\begin{equation}\label{zlefin}zW_\lambda\le\phi_k\,,\;\;\;\;\;\;\;\;s_izW_\lambda\le\phi_k\,.\end{equation}

\begin{lem}\label{L:findir} We have one of the following two cases, which are merged into a 
single case if $|\St|=1$: {\rm (1)} both statements in {\rm\eqref{zlefin}} hold for all $k$; 
{\rm (2)} the first statement holds for all $k$, while the second one only holds for $k=m$.
\end{lem} 
\begin{proof}
We start by assuming that $|\St|\ge 2$, and by showing that the first statement either holds 
for all $k$ or for no $k$. Based on \eqref{fincases}, this amounts to showing that, if $\phi_0=s_i\phi_m$ 
(which is the second case in the mentioned relation), then $zW_\lambda\le \phi_m$ implies 
$zW_\lambda\le\phi_0=s_i\phi_m$. This implication follows from Lemma \ref{L:Bruhatcoset}, 
since $s_i\phi_m<\phi_m$ and $s_iz>z$, so $s_izW_\lambda\ge zW_\lambda$. To complete the proof, 
it suffices to show that $zW_\lambda\le\phi_m$ implies $s_i zW_\lambda\le \phi_m$, for any $m\ge 0$. 
This is again justified by Lemma \ref{L:Bruhatcoset}, since we have $s_i\phi_m\le\phi_m$ in all possible cases. 
\end{proof}

Now let us consider the following Deodhar lifts, whenever the corresponding inequality in \eqref{zlefin} holds:
\[u_k:=\up(z,p_k)\,,\;\;\;\;\;\;\;\;u_k':=\up(s_iz,p_k)\,.\]
By Lemma \ref{L:findir}, all $u_k$ are defined, and either all $u_k'$ are defined or only $u_m'$. 
We will implicitly use this fact below.

\begin{proof}[Proof of Proposition {\rm \ref{P:pathlift}}] The main idea is to analyze 
systematically all possibilities regarding the relationships between $u_k$ and $u_l'$, 
for all $k$ and $l$. For each case, we indicate the corresponding case in the table \eqref{E:chart}.

By Proposition \ref{P:infin}, there are three main cases.
\begin{itemize}
\item
\textit{Case} S.0: $|\St|=1$.
\item
\textit{Case} S.1: $|\St|\ge 2$ and $u_0 W_\la = u_1 W_\la = \dotsm = u_m W_\la < s_i u_0 W_\la$.
\item
\textit{Case} S.2: $|\St|\ge 2$ and $u_0 W_\la < u_1 W_\la = \dotsm = u_m W_\la = s_i u_0 W_\la$.
\end{itemize}

\underline{\textit{Case} S.0.} We have the following two cases, by Lemma \ref{L:endstring}.

\textit{Case} S.0.1: $u_0=u_0'<s_i u_0$. This case leads to Case U.1.1.

\textit{Case} S.0.2: $u_0'=s_iu_0>u_0$. This case leads to Case U.2.1.

We now analyze cases S.1 and S.2, and start with some general observations. 
We have $u_0<s_i u_0$ and $u'_0<s_i u'_0$, by Lemma \ref{L:headlift}. This implies that
$u_0'\ne s_i u_0$, which is implicitly used several times below. 
By Lemma \ref{L:Pmiddle}, we have $u_1=u_2=\dotsm=u_{m-1}$, 
and (when they exist) $u'_1=u'_2=\dotsm=u'_{m-1}$. 
In addition, if $|\St|\ge 3$, then Lemma \ref{L:edgelift} gives $u_{m-1}=u_m$. 
Finally, Lemma \ref{L:endstring} gives $u_m=u'_m$, so we always have $u_1=u_2=\dotsm=u_{m}=u_m'$ in cases S.1 and S.2.

\underline{\textit{Case} S.1.} 
By Lemma \ref{L:edgelift}, 
we have
either $u_1=u_0$ or $u_1 = s_i u_0 > u_0$. Suppose the latter.
Then $s_i u_1 < u_1$. But $u_1=u_m$, so
$s_i u_m<u_m$, contradicting the assumption of Case S.1.
Therefore $u_1=u_0$ and we have
\[  u_0 = u_1 = \dotsm = u_m = u_m' < s_i u_0\,.\]

\textit{Case} S.1.1: $u_0',\,\dotsc,\,u_{m-1}'$ are not defined. This leads to Case U.1.1.

We may now assume that $u_0',\dotsc,u_{m-1}'$ are defined. We have the following two cases.


\textit{Case} S.1.2: $u'_1=u'_0$ or $|\St|=2$. It follows that $u_0'=u_1'=\ldots=u_{m-1}'$. 
If $u'_0=u_0$, then all $u_k$ and $u_l'$ coincide, and Case U.1.2 occurs. 
Otherwise, Case U.1.1 occurs for $x=u_0$, and Case U.3.1 for $x=u_0'$.

\textit{Case} S.1.3: $u'_1\ne u'_0$ and $|\St|\ge 3$. By Lemma \ref{L:edgelift}, we have $u'_1=s_i u'_0 > u'_0$.  
If $u'_0=u_0$, then we have Case U.1.3. Otherwise, we have Case U.1.1 for $x=u_0$, and Case U.3.2 for $x=u_0'$.

\underline{\textit{Case} S.2.} By Lemma \ref{L:edgelift} and the Case S.2 assumption, 
we have $u_0<u_1=s_i u_0$. Therefore, we have
\[u_0<u_1=\ldots=u_m=u_m'=s_i u_0\,.\]

\textit{Case} S.2.1: $u_0',\,\ldots,\,u_{m-1}'$ are not defined. This leads to Case U.2.1.

We may now assume that $u_0',\dotsc,u_{m-1}'$ are defined. We have the following two cases.

\textit{Case} S.2.2: $u'_1\ne u'_0$ or $|\St|=2$. If $|\St|\ge 3$, then Lemma \ref{L:edgelift} gives $u'_1=s_i u'_0>u_0'$. 
Assume first $u_0'=u_0$. The above facts imply $u_1=u_2=\ldots=u_m=u_1'=u_2'=\ldots=u_m'=s_i u_0$, so Case U.2.2 occurs. 
Now assume $u_0'\ne u_0$. Then Case U.2.1 occurs for $x=u_0$. Alternatively, for $x=u_0'$, 
we have Case U.3.2 if $|\St|\ge 3$, and Case U.3.1 if $|\St|=2$.

\textit{Case} S.2.3: $u_1'=u_0'$ and $|\St|\ge 3$. 
If $u_0'=u_0$, then Case U.2.3 occurs. If $u_0'\ne u_0$, then we have Case U.2.1 for $x=u_0$, and Case U.3.1 for $x=u_0'$.
\end{proof}

\subsection{On the proof of \eqref{comm1}}
\newcommand{\DD}{D}

The proof of the antidominant line bundle Chevalley rule \eqref{comm1} is omitted
as it is entirely analogous to that of \eqref{comm2}, but does not appear to formally follow from it.
The following is the analogue of Proposition \ref{P:pathlift}.
Note that the two propositions are related as follows: 
$\up$ switches with $\down$,  $\phi$ switches with $\iota$, 
the strings are reversed, and the Bruhat relations are dualized.
See Remark \ref{R:antidominant}.

\begin{prop} \label{P:downpathlift}
Let $w\in W$ be such that $sw < w$ and $w W_\la \ge \iota(t)$. Let $z=\down(w,t)$. Then $sz<z$ and
\begin{equation}\label{E:lift0}
\text{$\{\,\down(w,p) \mid p\in\St\,\}\subseteq\{z,sz\}$ and $\iota(p) \le w W_\la$ for $p\in\St$}.
\end{equation}
In terms of $\down(sw,\,\cdot\,)$ on $\St$, we have the following three disjoint cases:
\begin{align}
\label{E:lift1}
&\text{$\iota(p) \le swW_\la$ for $p\in\St$ and $\{\down(sw,p) \mid p\in\St\,\}\subseteq\{z,sz\}$} \\
\label{E:lift2} 
&\text{$\iota(h) \le swW_\la$, $\down(sw,h)\in\{z,sz\}$, and $\iota(p) \not\le swW_\la$ for $p\in\St\setminus\{h\}$} \\  
\label{E:lift3}
&\text{$\iota(p) \le swW_\la$ for $p\in\St$, $\down(sw,h)\in\{z,sz\}$, and 
$\{\down(sw,p) \mid p\in\St\setminus\{h\}\,\}\subseteq\{\tz,s\tz\}$} 
\end{align}
where $\tz\notin \{z,sz\}$ and $s\tz<\tz$.

More precisely, the following chart
gives the pairwise disjoint possibilities in terms of arbitrary $x$ with $sx<x$.
\begin{align}\label{E:downchart}
\begin{array}{|c||c|c|c||c|c|c|||c|c||} \hline
\Pd_{w,x}(\St) &    \multicolumn{3}{c||}{\St} & \multicolumn{3}{|c|||}{ t } &\multicolumn{2}{|c||}{ \emptyset } \\ \hline 
\Pd_{w,sx}(\St) &    \multicolumn{3}{c||}{\emptyset} & \multicolumn{3}{|c|||}{ \St\setminus t }
&\multicolumn{2}{|c||}{ \emptyset } \\ \hline\hline
\Pd_{sw,x}(\St) & h & \St & \{h,t\}& \emptyset  & t& \St\setminus h &\St\setminus h&t \\ \hline 
\Pd_{sw,sx}(\St)&\emptyset &\emptyset & m & h & \St\setminus t & h& \emptyset& m  \\ \hline\hline
|\St| & {\ge 1} & {\ge 2} & {\ge 3} & { \ge 1 }& {\ge 2}& {\ge 3} & {\ge 2 }&{\ge3} \\ \hline \hline
 \mathrm{Cases} & {\rm D}.1.1 &{\rm D}.1.2&{\rm D}.1.3 & {\rm D}.2.1 &{\rm D}.2.2 & {\rm D}.2.3 & {\rm D}.3.1 & {\rm D}.3.2 \\ \hline
\end{array}
\end{align}
Cases {\rm D}.a.b for $a\in\{1,2\}$ and $b\in\{2,3\}$ correspond to {\rm \eqref{E:lift1}} when $x=z$. Cases {\rm D.1.1} and 
{\rm D.2.1} correspond to \eqref{E:lift2} and \eqref{E:lift3} when $x=z$. Cases {\rm D.3.*} correspond to \eqref{E:lift3} when $x=\tz$.
\end{prop}

This concludes the proof of Theorem \ref{T:maintheorem}.

\section{The Chevalley formula in terms of the alcove model}\label{lchain}
In this section the Chevalley rules in Theorem \ref{T:maintheorem} are formulated
in terms of the alcove or $\la$-chain model of the first author and Postnikov \cite{LP,LP1}.
As a by-product, we obtain combinatorial descriptions of Demazure crystals and opposite
Demazure crystals \cite{Kas2} in terms of the alcove model, in Kac-Moody generality. 
In a mild difference of notation, in the definition of a $\la$-chain we use coroots consistently  
instead of roots, as in \cite{LP,LP1}.

\subsection{The alcove model and $\la$-chains}\label{S:alcmod} 
Let $\geh$ be a symmetrizable Kac-Moody algebra defined over $\C$.
Fix a dominant weight $\lambda$. By a root or coroot, 
we always mean a real root or coroot unless explicitly stated otherwise. 

An \emph{integral hyperplane} in $\La_\R=\R\otimes_{\Z}\La$ is one of the form
\begin{align}
  H_{\alpha,k} = \{ x\in \La_{\R}\mid \ip{\alpha}{x} = k\}
\end{align}
where $\alpha$ is a positive coroot and $k\in \Z$.

\begin{definition}\label{D:lambdahyperplane}
A \emph{$\la$-hyperplane} is an integral hyperplane $H_{\alpha,k}$
such that
\begin{align} \label{E:lambda-hyperplane}
0 \le k < \ip{\alpha}{\la}.
\end{align}
\end{definition}
By abuse of language, we will use the term ``$\la$-hyperplane'' to refer either to the pair $(\alpha,k)$
or the actual hyperplane $H_{\alpha,k}$. We call $k$ the height of $H_{\alpha,k}$.

\begin{rems}\label{R:hyperplane} (1) Consider the straight-line path in $\La_\R$ from $0$ to $\la$.
The $\la$-hyperplanes are precisely the integral hyperplanes that touch this path but do not
contain the endpoint $\la$. 

(2) If $\geh$ is infinite-dimensional then there are infinitely many positive
coroots, so that there are typically infinitely many $\la$-hyperplanes.
\end{rems}

\begin{definition}\label{D:lambdachain}
A \emph{$\la$-chain} is a total order on the set of $\la$-hyperplanes such that 
the following hold.
\begin{enumerate}
\item If $(\alpha,k),(\alpha,k')$ are $\la$-hyperplanes with $k<k'$, then $(\alpha,k)<(\alpha,k')$.
\item Given a $\la$-hyperplane $h=(\beta,k)$, a 
positive coroot $\alpha\ne\beta$, and an integer $m$ such that $\gamma=\alpha+m\beta$
is a positive coroot, we have
\[N_{<h}(\gamma)=N_{<h}(\alpha)+m N_{<h}(\beta)\]
where $N_{<h}(\eta)$ is the number of $\la$-hyperplanes less than $h$ with coroot $\eta$.
\end{enumerate}
\end{definition}

\begin{rem}
\label{R:original}
The original definition of a $\la$-chain in \cite{LP,LP1}
is obtained by forgetting the integer $k$ in each pair $(\alpha,k)$. More precisely, 
it involves the corresponding sequence of roots, while the counting condition 
uses the associated coroots. The integers $k$ are easily recovered by labeling the copies of
the coroot $\alpha$ by $0$ through $\ip{\alpha}{\la}-1$ in order of their appearance in the sequence of coroots.
\end{rem}

In \cite{LP1} a particular $\la$-chain is constructed.
It is described in the following proposition.
In particular $\la$-chains exist.

\begin{prop}\label{constr-lambdachain}\cite{LP1} Given a total order $I=\{1<2<\dotsm<r\}$
on the set of Dynkin nodes, one may express a coroot $\alpha=\sum_{i=1}^r c_i \alpha_i^\vee$
in the $\Z$-basis of simple coroots. Consider the
total order on the set of $\la$-hyperplanes defined by 
the lexicographic order on their images in $\Q^{r+1}$ under the map
\begin{equation}\label{E:stdvec}
(\alpha,k)\mapsto \frac{1}{\ip{\alpha}{\la}} (k,c_1,\ldots,c_r).
\end{equation}
This map is injective, thereby endowing 
the set of $\la$-hyperplanes with a total order, which is a $\la$-chain. We call it  the
lexicographic (lex) $\la$-chain. 
\end{prop}



For a finite root system, the definition of a $\lambda$-chain may be simplified.
The following is a characterization of the sequence of coroots (with repetition)
obtained from a $\la$-chain in the above sense, when the height $k$ of a $\la$-hyperplane $(\alpha,k)$ 
is forgotten. The heights are easily recovered due to condition (1) of Definition \ref{D:lambdachain}.

\begin{prop}\label{equivalence}\cite{LP1} 
Consider a finite root system and a finite sequence $(\beta_1,\beta_2,\dotsc,\beta_\ell)$ of positive coroots.
Then the following are equivalent:
\begin{enumerate}
\item The sequence of coroots is a $\la$-chain.
\item Each positive coroot $\alpha$ occurs exactly $\ip{\alpha}{\la}$ times in the sequence,
and for each triple of positive coroots $(\alpha,\beta,\gamma)$ with $\gamma=\alpha+\beta$, the subsequence 
restricted to copies of $\alpha$, $\beta$, and $\gamma$ 
is a concatenation of pairs $(\alpha,\gamma)$ and $(\beta,\gamma)$ (in any order).
\item There exists a {\em reduced alcove path}
$A_0=A_\circ\stackrel{-\beta_1}\longrightarrow \cdots
\stackrel{-\beta_l}\longrightarrow A_{l}=A_{-\lambda}$, in the sense of \cite{LP}.
\end{enumerate}
\end{prop}

Recall the notation $A\stackrel{\beta}\longrightarrow A'$, which means that the alcoves 
$A$ and $A'$ are separated by a hyperplane orthogonal to the coroot $\beta$, which points 
in the direction from $A$ to $A'$; $A_\circ$ is the fundamental alcove, and $A_\mu$ is 
its translation by $\mu$. Note that the first two properties are not equivalent in 
the Kac-Moody case. The reason is that there are broken $\beta$-strings of real roots 
through $\alpha$. Indeed, for an affine root system, consider the positive roots 
$\alpha=\overline{\alpha}+k\delta$ and $\beta=-\overline{\alpha}+m\delta$ with $k,m>0$, 
and $\overline{\alpha}$ a root of the corresponding non-affine root system. Note that $\alpha+2\beta$ 
is a positive real root, but $\alpha+\beta$ is an imaginary root.

In the sequel we will treat $\la$-chains either as sequences of positive coroots 
or as sequences of $\la$-hyperplanes, passing between the two definitions without further mention.

\subsection{The Chevalley formula} 
Let us fix a dominant integral weight $\lambda$ and an arbitrary $\lambda$-chain. 
For $\gamma\in\La$, let $t_\gamma$ be the operator on $\La$ given by translation by $\gamma$. 
Note that we are not working in an affine Weyl group but inside the group of automorphisms of
the lattice $\La$. For a coroot $\alpha$, let $\alpha^\vee$ be the associated root. 
By definition, if $\alpha = w \alpha_i^\vee$ for $w\in W$ and $i\in I$, then $\alpha^\vee=w \alpha_i$. 
For a coroot $\alpha$, let $s_\alpha$ act on $\La$ by the reflection
\begin{align}
 s_\alpha \cdot \mu = \mu - \ip{\alpha}{\mu} \alpha^\vee\,.
\end{align}
For a $\la$-hyperplane $h=(\alpha,k)$, we use the notation
\begin{align}
\label{E:k}
k_h &:= k \\
\label{E:m}
m_h &:= \ip{\alpha}{\la}-k \\
\label{E:s}
s_h &:= s_\alpha \\
\label{E:sh}
\hs_h &:= t_{k\alpha^\vee} s_\alpha \\
\label{E:st}
\ts_h &:= t_{m_h\alpha^\vee} s_\alpha.
\end{align}
The quantity $k_h$ is the number of hyperplanes with the same coroot $\alpha$
before $h$ in the given $\la$-chain. Note that $\hs_h$ is the reflection in $\La$
across the affine hyperplane $H_{\alpha,k}$.

\begin{definition}\label{D:adapted} For $z,w\in W$ with $z\le w$, we say that a sequence of $\la$-hyperplanes
$h_1,h_2,\dotsc,h_q$ (not necessarily increasing in some $\la$-chain) is \emph{$[z,w]$-adapted} if the coroots of the hyperplanes
are the associated coroots for a saturated Bruhat chain from $z$ to $w$:
\begin{equation}\label{brch}
z\lessdot zs_{h_1}\lessdot zs_{h_1}s_{h_2}\lessdot \ldots \lessdot zs_{h_1}s_{h_2} \dotsm s_{h_q}=w\,.
\end{equation}
We say that a sequence of $\la$-hyperplanes is \emph{$z$-adapted} if it is $[z,w]$-adapted for some $w\ge z$.
\end{definition}

\begin{thm}\label{kchev} Let $\lambda$ be a dominant weight. With respect to the lex $\la$-chain, we have
\begin{align}
\label{chevdom}
[L^{\la}]\,[\OO_z] &= \sum_{\substack{(h_1<\dotsm<h_q) \\ \text{$z$-adapted}}} e^{z\hs_{h_1}\dotsm\hs_{h_q}(\la)}\,[\OO_{zs_{h_1}\ldots s_{h_q}}]\,, \\
\label{chevantidom}
[L^{-\la}]\,[\OO_z] &= \sum_{\substack{(h_1>\dotsm>h_q) \\ \text{$z$-adapted}}}  (-1)^{q} \, e^{-z\ts_{h_1}\ldots\ts_{h_q}(\lambda)}\,[\OO_{zs_{h_1}\ldots s_{h_q}}]\,.
\end{align}
\end{thm}

There are also analogues of the two commutation formulas \eqref{comm2} and \eqref{comm1} in terms of the alcove model (see Example \ref{X:alcdom} and \ref{X:alcantidom}), which are similar to (\ref{chevdom}) and (\ref{chevantidom}). Formula (\ref{chevdom}) is proved in Sections \ref{refl-ord}, \ref{b-bruhat}, 
and \ref{ls-nonrec}, based on the corresponding formula (\ref{chdom}) in terms of LS paths. 
The proof of (\ref{chevantidom}) is based on (\ref{chantidom}), and is completely similar, cf. Remark \ref{lreford} and Section \ref{S:lexdecan}.

\begin{rem} Theorem \ref{kchev} yields a formula for multiplying by $[\OO_{X_{s_i}}]$ by noting that
\begin{equation}\label{divisor}[\OO_{X_{s_i}}]=1-e^{\Lambda_i} [L^{-\Lambda_i}]\,,\end{equation}
where $\Lambda_i$ is the $i$-th fundamental weight.
\end{rem}

\begin{conj} Theorem {\rm \ref{kchev}} holds for any $\lambda$-chain.
\end{conj}

In fact, as the finite-type $K$-Chevalley formula in \cite{LP} works for an arbitrary weight $\lambda$, 
Theorem \ref{kchev} should extend to an arbitrary $\lambda$ as well. This would require the 
corresponding generalization of the concept of a $\lambda$-chain in the Kac-Moody case,
which is non-trivial (in the finite case, we can use condition (3) in Proposition \ref{equivalence}
as a definition, and this is straightforward to extend to an arbitrary $\lambda$). In particular, 
a $\lambda$-chain will now have positive and negative roots, and the condition that the coroot $\alpha$ appears 
$\ip{\alpha}{\la}$ times must be replaced by the requirement that the number of occurrences of 
an arbitrary root $\alpha$ is the maximum of $0$ and $\ip{\alpha}{\la}$. For instance, we can define a 
$w(\lambda)$-chain, for $\lambda$ dominant and $w$ in the Weyl group (so that $w(\lambda)$ 
is in the Tits cone) essentially by applying $w$ to a $\lambda$-chain for dominant $\lambda$.
Also note that the negative reverse of a $\la$-chain should be a $(-\la)$-chain.

Assuming an arbitrary $\lambda$, in order to extend the proof techniques in \cite{LP}, 
it is not enough to consider only the $\lambda$-chains mentioned above. 
It turns out that it is necessary to uniformly prove a Chevalley formula in which the adapted sequences are chosen 
from a more general set of hyperplanes; these are obtained from the $\la$-chains above via a ``folding'' procedure, see \cite[Section 5]{LP1}. 
In such a ``folded $\lambda$-chain'', the same hyperplane can appear several times. 
Also note that, for general $\lambda$, a Chevalley formula will have cancellations, 
even if it is based on a minimal $\lambda$-chain (i.e., having no repeated hyperplanes). This is not the case when $\lambda$ 
is dominant or antidominant, as the formulas in Theorem \ref{kchev} have no cancellations.

\subsection{Reflection orders}\label{refl-ord} The proof of the alcove model 
Chevalley formula begins with some considerations regarding Dyer's {\em reflection orders} \cite{Dyer}. 
Let $W$ be the Weyl group of a symmetrizable Kac-Moody algebra $\geh$.

For the entire proof the dominant weight $\lambda$ is fixed.
Let $J=\{i\in I\mid s_i\cdot\lambda = \lambda\}$. Then $W_\lambda = W_J$ is the stabilizer of $\lambda$.
Let $W^\lambda = W^J$. Denote by $\Pcp$ the set of positive real coroots for $\geh$.
The Bruhat graph on $W$ is the graph with vertex set $W$ and a directed
edge from $v$ to $w$ if $v$ is covered by $w$. This edge is labeled by the unique element $\alpha\in \Pcp$
such that $w = v s_\alpha$. This is denoted $v\stackrel{\alpha}\longrightarrow w$. 

\begin{definition}\cite{Dyer} A reflection order is a total order on $\Pcp$ satisfying the following property: for every $\alpha,\beta\in\Pcp$ and $a,b\in{\mathbb R}_{>0}$ such that $a\alpha+b\beta\in\Pcp$, we have
\begin{equation}\label{ord3}\alpha<a\alpha+b\beta<\beta\;\;\;\;\mbox{or}\;\;\;\;\beta<a\alpha+b\beta<\alpha\,.\end{equation}
\end{definition}

The above definition is one of several equivalent ones. The main result related to reflection orders is the following one, known as the {\em EL-shellability} of the Bruhat order; we state only the part of this result that we need.

\begin{prop}\label{shell}\cite{Dyer}
Let $v\le w$ in Bruhat order. Then for any reflection order, 
there exists a unique saturated Bruhat chain from $v$ to $w$ with labels which increase in the reflection order.
\end{prop}

We now define a total order $<_\la$ on $\Pcp$ which depends on $\la$. 
The bottom of the order $<_\la$ consists of the coroots $\alpha\in\Pcp$
such that $\ip{\al}{\la}>0$. For two such coroots $\alpha$ and $\beta$, define $\alpha<\beta$ if
$(\alpha,0) < (\beta,0)$ in the lex $\la$-chain.
This forms an \emph{initial section \cite{Dyer}} of $<_\la$.
The top of the order $<_\la$ consists of the 
$\alpha\in\Pcp$ orthogonal to $\la$; such $\alpha$ form the positive coroots for the Weyl group $W_\la$,
and one may use any reflection order for them.

\begin{lem} \label{L:<lambdareflorder}
 The total order $<_\la$ is a reflection order on $\Pcp$.
\end{lem}
\begin{proof} Consider $\alpha<_\la\beta$ in $\Pcp$, and assume that $a\alpha+b\beta\in\Pcp$ 
for some $a,b\in \R_{>0}$. Let $c_\alpha=\ip{\alpha}{\la}$ and $c_\beta=\ip{\beta}{\la}$. 
It suffices to show \eqref{ord3}.
This holds if $c_\alpha=c_\beta=0$ since a reflection order was used for the positive coroots of $W_\la$.
We represent vectors in the basis of simple coroots as tuples of coordinates, using the 
chosen order on the simple coroots. Assume first that $c_\alpha,c_\beta>0$. It suffices to show that
\[\frac{\alpha}{c_\alpha}<\frac{a\alpha+b\beta}{ac_\alpha+bc_\beta}<\frac{\beta}{c_\beta}\]
in lexicographic order; here the inequality between the first vector and the last one is known, 
as it expresses $\alpha<_\lambda\beta$. The proof is completed by noting that the middle fraction can be written
\begin{align*}
 c\,\frac{\alpha}{c_\alpha} + (1-c)\, \frac{\beta}{c_\beta}\,,
\qquad\text{where}\qquad\text{$c=\frac{a c_\alpha}{a c_\alpha+bc_\beta}\in(0,1)$.}
\end{align*}
The remaining case is $c_\alpha>0$, $c_\beta=0$. It suffices to show
\[\frac{\alpha}{c_\alpha}<\frac{a\alpha+b\beta}{ac_\alpha}=\frac{\alpha}{c_\alpha}+\frac{b}{ac_\alpha}\beta\,,\]
which is obvious.
\end{proof}

\subsection{The $b$-Bruhat order}\label{b-bruhat}
The definitions in this section depend on the fixed dominant weight $\la$ and a fixed
rational number $b$. Let 
\begin{align*}
\Phi_b=\{\alpha\in\Pcp \mid b \,\ip{\alpha}{\la}\in\Z\}\qquad\qquad
\Phi_b^*=\Phi_b\setminus\{\alpha\in\Pcp \mid \ip{\alpha}{\la}=0\}. 
\end{align*}

\begin{definition}\cite{LS,St} The {\em $b$-Bruhat order} $\le_b$ on $W$ is defined by the Bruhat covers 
$v\overset{\alpha}{\rightarrow}{w}$ with $\alpha\in \Phi_b$.
\end{definition}

Clearly, for $b\in\Z$, $\Phi_b=\Pcp$ and $\le_b$ is the Bruhat order. For simplicity, we use the term Bruhat 
(resp. $b$-Bruhat) chain for a saturated chain in Bruhat (resp. $b$-Bruhat) order.

\begin{lem}\label{L:bchains} Suppose that $v\le_b w$. Then every Bruhat chain 
from $v$ to $w$ is a $b$-Bruhat chain. In particular, the Bruhat interval $[v,w]$
coincides with the corresponding $b$-Bruhat interval. 
\end{lem}

\begin{proof} 
Bj\"{o}rner and Wachs \cite{BW} showed that the order complex of the open Bruhat interval $(v,w)$ 
is a combinatorial sphere. 
It follows that any two Bruhat chains from $v$ to $w$
can be connected with a sequence of Bruhat chains from $v$ to $w$ such that any two adjacent chains
differ in exactly one position. By hypothesis, there is a $b$-Bruhat chain from $v$ to $w$.
It must be shown that any Bruhat chain from $v$ to $w$ is a $b$-Bruhat chain.
By the above connectedness and induction,
it suffices to prove this under the assumption that the $b$-Bruhat chain and the
Bruhat chain differ at exactly one position.
This involves studying the relationships of labels on an interval of length $2$ in a dihedral subgroup.

Consider a dihedral group with generators $s_\alpha,s_\beta$. Consider an interval of length 2 with minimum $u$ 
and two chains labeled by coroots $(\gamma,\delta)$ and $(\varepsilon,\phi)$. 
We may assume that $(\gamma,\delta)$ are $b$-Bruhat covers, 
and must show that $(\varepsilon,\phi)$ are $b$-Bruhat covers. It suffices to show that the latter coroots 
are integer linear combinations of the former; we denote this property by $(\gamma,\delta)\rightarrow (\varepsilon,\phi)$. 
For $k\ge 0$, let 
\begin{equation}\label{gab}\gamma_{2k}=(s_\alpha s_\beta)^k(\alpha)\,,\;\;\;\;\gamma_{2k+1}=(s_\alpha s_\beta)^k
s_\alpha(\beta)\,.
\end{equation}
The typical pairs $(\gamma,\delta)$, $(\varepsilon,\phi)$ are of the following form, 
where in each case we indicate the non-trivial property $(\gamma,\delta)\rightarrow (\varepsilon,\phi)$ 
to be proved:
\begin{itemize}
\item $(\alpha,s_\alpha(\gamma))\rightarrow (\gamma,\alpha)$, where $u=\ldots s_\alpha s_\beta$;
\item $(\gamma_{i},\gamma_{i+1})\rightarrow (\beta,\alpha)$, where $u=\ldots s_\beta s_\alpha$ 
and $\ell(u)=i$, for $i=0,1,\ldots$.
\end{itemize}
The first property is obvious. For the second one, note that the reflections 
$s_{\gamma_i}$ and $s_{\gamma_{i+1}}$ generate the dihedral group (although they are not Coxeter generators). So, by
\eqref{gab}, $\alpha$ and $\beta$ can be obtained by applying a certain sequence 
consisting of the new generators to $\gamma_i$ and another such sequence to $\gamma_{i+1}$. This concludes the proof. 
\end{proof}

\begin{lem}\label{L:deob} Consider $\sigma,\tau$ in $W^\lambda$ and $w_\lambda,w_\lambda'$ in $W_\lambda$. 
We have $\sigma w_\lambda\le_b \tau w_\lambda'$ if and only if $\sigma\le_b \tau$ and 
$\up(\sigma w_\lambda,\tau W_\lambda)\le \tau w_\lambda'$.
\end{lem}

\begin{proof} For $b=0$, this is just a restatement of 
Proposition \ref{P:Deodhar} (2) combined with Proposition \ref{P:coset} (4). 
Assuming that the two equivalent statements hold for $b=0$, it remains 
to show for an arbitrary $b$ that $\sigma w_\lambda\le_b \tau w_\lambda'$ 
if and only if $\sigma\le_b \tau$. For the ``if'' statement, note that 
$\tau\le_b\tau w_\lambda'$, since the labels of any corresponding Bruhat chain 
are orthogonal to $\lambda$. Thus $\sigma\le_b \tau w_\lambda'$, and the conclusion follows 
from Lemma \ref{L:bchains}, since $\sigma\le\sigma w_\lambda\le \tau w_\lambda'$. 
The ``only if'' statement is completely similar.
\end{proof}

\begin{lem}\label{L:mainb} Consider $\sigma,\tau\in W^\lambda$ and $w_\lambda, w_\lambda'\in W_\lambda$ 
such that $\sigma w_\lambda\le \tau w_\lambda'$. We have $\sigma\le_b \tau$ and 
$\tau w_\lambda'=\up(\sigma w_\lambda,\tau W_\lambda)$ if and only if the unique Bruhat chain 
from $\sigma w_\lambda$ to $\tau w_\lambda'$ with labels which are increasing
with respect to the reflection order $<_\lambda$ has all its labels in $\Phi_b^*$. 
\end{lem}
\begin{proof} For the ``only if'' statement, by Proposition \ref{shell}
and Lemma \ref{L:<lambdareflorder} there is a unique saturated chain from
$\sigma w_\la$ to $\tau w_\la'$ that has labels increasing with respect to the
reflection order $<_\la$. By Lemma \ref{L:deob}, we have $\sigma w_\lambda\le_b \tau w_\lambda'$.
So the above Bruhat chain has all its labels in $\Phi_b$, by Lemma
\ref{L:bchains}. If it had labels orthogonal to $\lambda$, all of them would have to appear at the end of the chain, 
cf. the definition of the reflection order $<_\lambda$. But then $\tau w_\lambda'$ 
cannot be smallest among the elements of $\tau W_\lambda$ which are greater or 
equal to $\sigma w_\lambda$, contradicting the minimality of the lift.
For the ``if'' statement, note first that $\sigma\le_b \tau$ 
and $\up(\sigma w_\lambda,\tau W_\lambda)\le \tau w_\lambda'$, by Lemma \ref{L:deob}. 
If the latter inequality is strict, construct a chain from 
$\sigma w_\lambda$ to $\tau w_\lambda'$ by concatenating the chains determined 
by $<_\lambda$ from $\sigma w_\lambda$ to $\up(\sigma w_\lambda,\tau W_\lambda)$ 
and from $\up(\sigma w_\lambda,\tau W_\lambda)$ to $\tau w_\lambda'$. The first chain 
has all its labels in $\Phi_b^*$, by the ``only if'' statement just proved. 
Therefore, the constructed chain has increasing labels (with respect to $<_\la$), so it coincides with the one in the statement of the lemma (by the uniqueness property in Proposition \ref{shell}). The fact that the latter chain has all 
its labels in $\Phi_b^*$ is thus contradicted.
\end{proof}

\begin{rem}\label{lreford} There is a version of Lemma \ref{L:mainb} for the Deodhar lift ``down'', 
which is used in the proof of \eqref{chevantidom}, by analogy with the proof in Section \ref{ls-nonrec}. 
This version is based on a reflection order in which the roots corresponding to $W_\lambda$ form an
initial section, as opposed to being larger than the other roots, as in the case of $<_\lambda$.
\end{rem}

\subsection{The proof of \textrm{\eqref{chevdom}}} \label{ls-nonrec}
In the sequel, the total order on the $\la$-hyperplanes is the lex $\la$-chain. Let $\Inc_{w,z}^\la$ be the set of lex-increasing sequences of $\la$-hyperplanes which are $[z,w]$-adapted.

The dominant weight alcove model Chevalley formula \eqref{chevdom} is an immediate consequence of
the corresponding LS path formula \eqref{chdom} due to the following bijection.
Refer to \eqref{E:uppaths} and Definition \ref{D:adapted} for the key definitions.

\begin{prop} \label{P:LSincreasingalcove} There is a bijection 
$\Inc_{w,z}^\la \cong \Pu_{w,z}^\la$ that preserves weights: if
$H=(h_1<h_2<\dotsm<h_q)\in \Inc_{w,z}^\la$ maps to $p\in \Pu_{w,z}^\la$, then 
\begin{align}\label{E:incweight}
  p(1) = \wt(H) := z \,\hs_{h_1} \hs_{h_2}\dotsm \hs_{h_q} \cdot \la\,.
\end{align}
\end{prop}

In the remainder of this subsection we fix a Bruhat interval $[z,w]$, construct a map 
$\Inc_{w,z}^\la\to\Pu_{w,z}^\la$, then a map
$\Pu_{w,z}^\la\to \Inc_{w,z}^\la$, show that the maps are inverses,
and finally that \eqref{E:incweight} is satisfied.

The proof starts by recalling the non-recursive description of 
LS paths \cite{LS} (see also \cite{St}). A path $p\in\Ta^\la$ is given by two sequences
\begin{equation}\label{E:lsste}
0 = b_1<b_2<b_3<\ldots<b_m<b_{m+1}=1 \qquad\phi(p)=\sigma_1<_{b_2}<\sigma_{2}<_{b_3}\ldots <_{b_m}\sigma_m=\iota(p)\,,
\end{equation}
where $b_i\in \Q$ and $\sigma_i\in W/W_\la$, cf. Sections \ref{mainresls} and \ref{lslifts}. 
This data encodes the sequence of vectors 
$(b_{m+1}-b_m)\sigma_m\cdot\la$, $\dotsc$, $(b_3-b_2)\sigma_2\cdot\la$, $(b_2-b_1)\sigma_1\cdot\la$.

Let $H=(h_1<h_2<\dotsm<h_q) \in \Inc_{w,z}^\la$. The \emph{relative height} $\rht(h)$
and \emph{relative coheight} $\rcht(h)$ of the $\la$-hyperplane $h=(\alpha,k)$ are defined by
\begin{align}
\label{E:relht}
\rht(h) &= \dfrac{k}{\ip{\alpha}{\la}}\\
\label{E:relcoht}
\rcht(h) &= 1 - \rht(h).
\end{align}
By the definition of the lex $\la$-chain, the relative heights of $h_1,h_2,\dotsc,h_q$ 
form a weakly increasing sequence in $[0,1) \cap \Q$.
Let us call the sequence of distinct nonzero relative heights $0<b_2<b_3<\dotsm<b_m<1$, and let $b_1=0$.
Note that the relative height $0$ is treated differently;
in part, this is because the $0$-Bruhat order is the same as the Bruhat order.
For $j\ge 1$, let $I_j$ be the subinterval of $H$ consisting of the elements of relative height $b_j$;
these sets are all nonempty, except perhaps $I_1$.

Since $H\in\Inc_{w,z}^\la$, there is a saturated Bruhat chain
\begin{align} \label{E:thechain}
 z \lessdot z s_{h_1} \lessdot zs_{h_1} s_{h_2} \lessdot \dotsm \lessdot z s_{h_1}\dotsm s_{h_q}=w.
\end{align}
We pick out some of the elements in the above chain by multiplying by groups of reflections
given by the subsets $I_j$. For $0\le j\le m$, define the Weyl group elements
\begin{align}
\label{E:z}
  z_j &= z \prod_{h\in I_1\cup I_2\cup\dotsm\cup I_j}^{\longrightarrow}
\end{align}
where the (non-commutative) product over $h$ occurs from left to right in the 
order $h_1<h_2<\ldots$. In particular $z_0=z$. Let 
\begin{align}\label{E:ztosigma}
  \sigma_j = z_j W_\la \in W/W_\la\qquad\text{for $1\le j\le m$.}
\end{align}
The required map is $H\mapsto (b_2,\dotsc,b_m);(\sigma_1,\dotsc,\sigma_m)$.
This map is clearly well-defined: the $b_j$-Bruhat condition
is satisfied by the definition of relative height.

The inverse map is constructed using Deodhar lifts and the EL-shellability of the Bruhat order. We begin with an
LS path $p\in\Pu_{w,z}^\la$ in the form \eqref{E:lsste}. By definition, we have 
$zW_\la \le \phi(p)=\sigma_1$. Letting $z_0=z$, define the lifts
\begin{equation}\label{ups}
  z_j = \up(z_{j-1},\sigma_j)\qquad\text{for $j=1,2,\dotsc,m$.}
\end{equation}
By definition we have $z_m = \up(z,p)=w$. 
By the ``only if'' part of Lemma \ref{L:mainb}, for each $1\le j\le m$ there is a unique
$b_j$-Bruhat chain from $z_{j-1}$ to $z_j$ with labels in $\Phi_{b_j}^*$, 
which are increasing with respect to the reflection order $<_\lambda$. 
Let us replace each label $\beta$ in this chain with the pair $(\beta,b_j \ip{\beta}{\la})$, 
where the second component is in $\{0,1,\dotsc, \ip{\beta}{\la}-1\}$ as $\beta\in\Phi_{b_j}^*$. 
Thus, each such pair is a $\la$-hyperplane. 
By concatenating these chains we obtain a Bruhat chain from $z$ to $w$, together with
a sequence of $\la$-hyperplanes $H$. This defines the map $p\mapsto H$.

\begin{proof}[Proof of Proposition {\rm \ref{P:LSincreasingalcove}}]
The forward map was seen to be well-defined. For the well-definedness of the
inverse map, note first that the sequence of $\la$-hyperplanes is $[z,w]$-adapted by definition.
It is also lex-increasing: the relative heights of the 
$\la$-hyperplanes weakly increase by construction,
and within the same relative height the $\la$-hyperplanes increase because 
of the compatibility of the reflection order $<_\la$ on coroots with the lex $\la$-chain.

To show that the two maps are mutually inverse, the crucial fact to check is that
the forward map composed with the backward one is the identity.
This follows from the ``if'' part of Lemma \ref{L:mainb}, after recalling from the above discussion that the order of the
$\la$-hyperplanes with fixed relative height is given by the reflection order $<_\la$ on the coroots. 

Finally, we check the weight condition \eqref{E:incweight}. If $h=(\alpha,k)$ is a $\la$-hyperplane of relative height $b$, we have
\begin{align*}
  \hs_h \cdot b \la &= k \alpha^\vee + b (\la - \ip{\alpha}{\la} \alpha^\vee) \\
  &= b\la + (k - b \ip{\alpha}{\la})\alpha^\vee \\
  &= b\la
\end{align*}
by the definitions \eqref{E:sh} of $\hs$ and \eqref{E:relht} of relative height.
This given, using our previous conventions on ordered products and $b_{m+1}=1$, we have
\begin{align*}
  \left(\prod_{h\in I_m}^{\longrightarrow} \hs_h \right)\cdot \la = b_m\la + (b_{m+1}-b_m) \left(\prod_{h\in I_m}^{\longrightarrow} s_h \right)\cdot \la\,.
\end{align*}
Writing $b_m\la = b_{m-1}\la+(b_m-b_{m-1})\la$ and applying $\prod_{h\in I_{m-1}} \hs_h$ to the above equation,
we obtain
\begin{align*} 
\left(\prod_{h\in I_{m-1}\cup I_m}^{\longrightarrow} \hs_h \right)\cdot \la &= 
b_{m-1}\la + (b_m-b_{m-1})\left(\prod_{h\in I_{m-1}}^{\longrightarrow} s_h \right)\cdot \la + 
(b_{m+1}-b_m)\left(\prod_{h\in I_{m-1}\cup I_m}^{\longrightarrow} s_h \right)\cdot  \la\,.
\end{align*}
Iterating this until $\prod_{h\in I_1} \hs_h$ has been applied, then applying $z$,
and using the definition \eqref{E:z}, we obtain
\begin{align*}
  z \left(\prod_{h\in H}^{\longrightarrow} \hs_h \right)\cdot \la &= \sum_{j=1}^m (b_{j+1}-b_j)z_j \cdot \la \\
  &= \sum_{j=1}^m (b_{j+1}-b_j)\sigma_j \cdot \la \\
  &= p(1)\,.
\end{align*}
Here the last equality follows from \cite[(8.3)]{St}, or the fact that
$p$ consists of following the vectors $(b_{j+1}-b_j)\sigma_j\cdot \la$
for $j$ going from $m$ down to $1$.
\end{proof}

\begin{proof}[Proof of {\rm \eqref{chevdom}}] Immediate based on \eqref{chdom} and Proposition \ref{P:LSincreasingalcove}.
\end{proof}

\subsection{Lex-decreasing analogue}\label{S:lexdecan} Let $\Dec_{w,z}^\la$ be the set of lex-decreasing sequences of $\la$-hyperplanes which are $[z,w]$-adapted.

\begin{prop} \label{P:LSdecreasingalcove} There is a bijection 
$\Dec_{w,z}^\la\cong \Pd_{w,z}^\la$ that is weight preserving: if
$H=(h_1>h_2>\dotsm>h_q)\in \Dec_{w,z}^\la$ maps to $p\in \Pd_{w,z}^\la$, then 
\begin{align}\label{E:decweight}
  p(1) = \widetilde{\wt}(H) := z \,\ts_{h_1} \ts_{h_2}\dotsm \ts_{h_q} \cdot \la\,.
\end{align}
\end{prop}

The proof and constructions are analogous to those in the lex-increasing case.
Some details are given below.

Let $H=\{h_1>h_2>\dotsm>h_q\}\in \Dec_{w,z}^\la$.

Since $H$ is lex-decreasing, the sequence $\{ \rcht(h_i)  \}$ of relative \emph{coheights} (see \eqref{E:relcoht})
is a weakly increasing sequence in $(0,1] \cap\Q$.
Let their distinct values other than $1$ be $0<b_2<b_3<\dotsm<b_m<1$ and let $b_{m+1}=1$.
Let $I_j$ be the elements in $H$ of relative coheight $b_j$, for $2\le j\le m+1$.
Then $I_j$ is nonempty, except possibly for $I_{m+1}$. 

For $1\le j\le m+1$ let
\begin{align}
  w_j = z \prod_{h\in I_2\cup I_3\cup\dotsm\cup I_j}^{\longrightarrow} s_h.
\end{align}
In particular, $w_1 = z$ and $w_{m+1}=w$. For $2\le j\le m+1$, the
elements of $I_j$ define a $(1-b_j)$-Bruhat chain from $w_{j-1}$ up to $w_j$.
But by definition, the $b$-Bruhat order coincides with the $(1-b)$-Bruhat order.
So we have a $b_j$-Bruhat chain.

Define $\sigma_j = w_j W_\la$, for $1\le j\le m$.

Then $\sigma_1 <_{b_2} \sigma_2 <_{b_3} \dotsm <_{b_m} \sigma_m$
defines an element $p \in\Ta^\la$ and $ \sigma_m \le w W_\la$. 

\begin{lem} We have $\down(w_{j+1},\sigma_j)=w_j$ for $1\le j\le m$.
\end{lem}

It follows that $p\in \Pd_{w,z}^\la$. 

The definition of the inverse map is similar to that for the lex-increasing case.
One must use the fact that the dual of a reflection order is again a reflection order.

\subsection{Consequences of the bijections}

Let $\la$ be a dominant weight and $z,w\in W$. The \emph{Demazure module} of lowest weight $w\lambda$
is the module $U_q(\mathfrak{b}) v_{w\lambda}$, where $U_q(\mathfrak{b})$ is the upper triangular part
of the quantum universal enveloping algebra $U_q(\geh)$, and $v_{w\lambda}$ is a vector of extremal weight $w\lambda$
in the highest weight $U_q(\geh)$-module of highest weight $\lambda$.
The \emph{opposite Demazure module} of highest weight $z\lambda$ is the module 
$U_q(\mathfrak{b}_-) v_{z\lambda}$, where $U_q(\mathfrak{b}_-)$ is the lower triangular part of $U_q(\geh)$.

\begin{thm} \cite{Kas2,Li1} 
{\rm (1)} The Demazure crystal of lowest weight $w\lambda$ is given by the subcrystal of $\Ta^\la$
consisting of paths $p$ with $\iota(p) \le w\lambda$.

{\rm (2)} The opposite Demazure crystal of highest weight $z\lambda$ is given by the subcrystal of $\Ta^\la$
consisting of paths $p$ with $\phi(p) \ge z\lambda$.
\end{thm}

Via Propositions \ref{P:LSincreasingalcove} and \ref{P:LSdecreasingalcove},
we have the following realization of Demazure and opposite Demazure crystals in terms of the alcove model,
whose crystal structure was defined in \cite{LP1} in Kac-Moody generality.

\begin{cor} \label{C:demopdem} {\rm (1)} The Demazure crystal of lowest weight $w\lambda$ is given by set of lex-decreasing
sequences of $\la$-hyperplanes adapted to intervals of the form $[z,w]$, for some $z\le w$.

{\rm (2)} The opposite Demazure crystal of highest weight $z\la$ is given by the
set of lex-increasing sequences of $\la$-hyperplanes which are $z$-adapted, that is, $[z,w]$-adapted
for some $w\ge z$.
\end{cor}

\begin{rem} In the case of finite root systems, the statement of Corollary \ref{C:demopdem} 
(1) appears in \cite{LP}[Corollary 6.5].
\end{rem}

\section{Examples}\label{affine}

\subsection{Dominant weight} 
\label{SS:domLSexample}
Consider the affine Lie algebra $\geh$ of type $A_{3-1}^{(1)}$, the dominant weight $\la=\La_0+\La_1$,
and $w=s_0s_1s_2s_1=s_0s_2s_1s_2$. In this subsection we will obtain all 
coefficients $b^w_{z \la}$ for $\la$ and $w$ fixed as above and varying $z\le w$.
This will be done first by LS paths and then by the alcove model.

\begin{ex} \label{X:LSup}
We illustrate the dominant weight LS path Chevalley formula \eqref{comm2}.
The stabilizer subgroup $W_\la$ is $W_J$ with $J=\{2\}$,
and the representative of $w$ in $W^J$ is $s_0s_2s_1$.

\begin{figure}
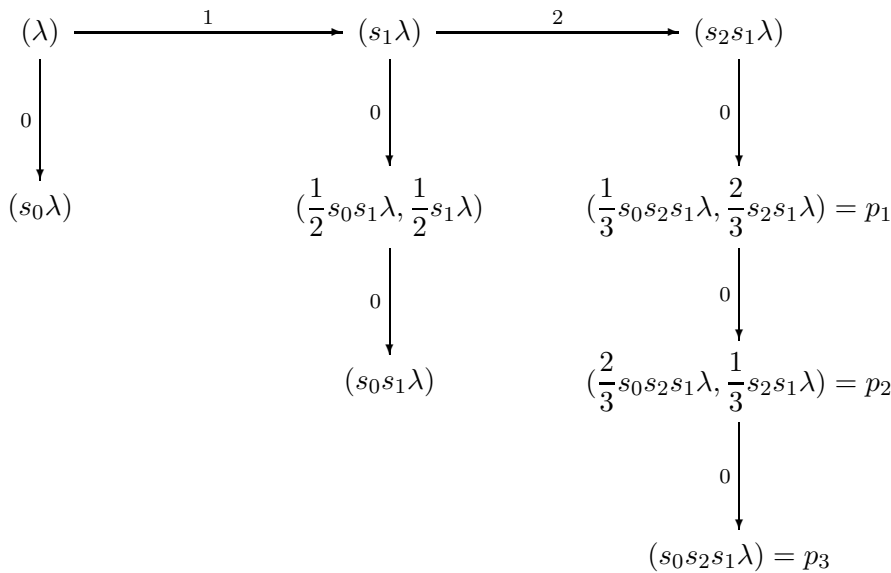

\begin{align*}
\begin{diagram}
\node{(\la) } \arrow{s,t,->}{0} \arrow{e,t,->}{1} \node{(s_1\la)} \arrow{s,t,->}{0} \arrow{e,t,->}{2} \node{(s_2s_1\la)}
\arrow{s,t,->}{0} \\
\node{(s_0\la)} \arrow{e,t,!}{} \node{(\frac{1}{2}s_0s_1\la,\frac{1}{2}s_1\la)} \arrow{s,t,->}{0} \arrow{e,t,!}{}
\node{(\frac{1}{3}s_0s_2s_1\la,\frac{2}{3}s_2s_1\la)=p_1} \arrow{s,t,->}{0} \\
\node[2]{(s_0s_1\la)} \arrow{e,t,!}{} \node{(\frac{2}{3}s_0s_2s_1\la,\frac{1}{3}s_2s_1\la)=p_2} \arrow{s,t,->}{0} \\
\node[3]{(s_0s_2s_1\la)=p_3}
\end{diagram}
\end{align*}
\label{F:XdomLS}
\caption{Demazure crystal}
\end{figure}

The Demazure subcrystal of $\Ta^\la$ of lowest weight $w\la$ is given in Figure \ref{F:XdomLS}.
The vertices are LS paths, which are written as sequences of vectors in $\Q\otimes  \La$.
The crystal operators $f_i$ are denoted by arrows labeled $i$. Recall that traversing
an arrow $i$ changes weight by $-\alpha_i$.
In the $K$-Chevalley formula \eqref{comm2}, any LS path $p\in \Pu^\lambda_{w,z}$
must have initial direction $w\lambda=s_0s_2s_1\lambda$. Here the set of such paths is
$\{p_1, p_2, p_3\}$. For each of these paths $p$, the set of $z$ such that $w=\up(z,p)$ is given as follows.
For $p_1$ and $p_2$, we have $z\in \{s_1s_2, s_1s_2s_1\}$, and for $p_3$ we have $z\in\{s_1s_2,s_1s_2s_1,s_0s_1s_2,s_0s_1s_2s_1\}$. This yields a total of $8$ pairs $(z,p)$. We have
\begin{align*}
b^w_{z \la} = \begin{cases} 
e^{\wt(p_3)}& \text{for $z\in \{s_0s_1s_2,s_0s_1s_2s_1\}$} \\
e^{\wt(p_1)}+e^{\wt(p_2)}+e^{\wt(p_3)}& \text{for $z\in \{s_1s_2,s_1s_2s_1\}$.}
\end{cases}
\end{align*}

Let us see for $p_1$ why $z=s_1s_2$ works, and $z=s_1$ does not. 
With $z=s_1s_2$, the Deodhar lifts (of the steps of the path $p_1$ listed in reverse order, with
$z$ at the front) are $(s_1s_2;s_1s_2s_1, s_0s_1s_2s_1)$, and this sequence ends at $w$.
For $z=s_1$, the Deodhar lifts form the sequence $(s_1;s_2s_1,s_0s_2s_1)$, which does not end at $w$.
\end{ex}

\begin{ex}\label{X:alcdom} Using $\geh$, $\la$, and $w$ as above,
we work out the dominant weight alcove model Chevalley formula \eqref{chevdom},
using the lex $\la$-chain based on the order $0<1<2$ on the Dynkin node set.

\begin{figure}
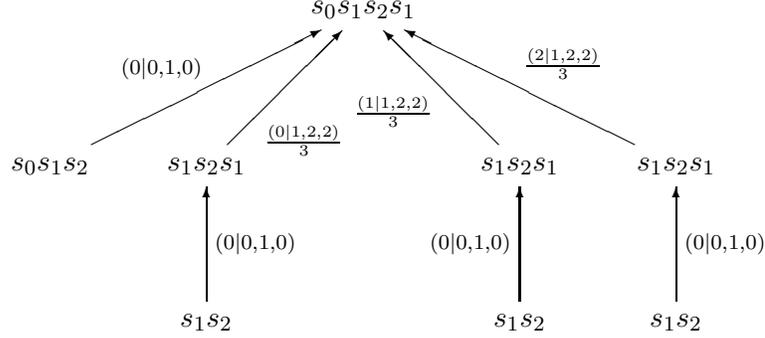

\begin{align*}
\begin{diagram}
\node[3]{s_0s_1s_2s_1}  \\
\node{s_0s_1s_2}\arrow{ene,t,->}{(0|0,1,0)}  \node{s_1s_2s_1}\arrow{ne,b,1}{\frac{(0|1,2,2)}{3}}
\node[2]{s_1s_2s_1}\arrow{nw,b}{\frac{(1|1,2,2)}{3}} \node{s_1s_2s_1}\arrow{wnw,t,->}{\frac{(2|1,2,2)}{3}}
 \\
\node[2]{s_1s_2} \arrow{n,b,->}{(0|0,1,0)} \node[2]{s_1s_2} \arrow{n,t,->}{(0|0,1,0)} 
\node{s_1s_2} \arrow{n,b,->}{(0|0,1,0)} 
\end{diagram}
\end{align*}
\label{F:DomTreeAlc}
\caption{Dominant tree}
\end{figure}

Let $\la$ and $w$ be fixed. We compute $b^w_{z \la}$ for varying $z\le w$ using $\Inc_{w,z}^\la$.
These sequences may be efficiently generated by a depth-first search, which generates
the tree in Figure \ref{F:DomTreeAlc}. This tree is rooted at $w$, its vertices are Weyl group elements (which can be repeated),
its edges are labeled by $\la$-hyperplanes, and its branches going towards the root
are lex-increasing. The vertices are in bijection
with $\bigsqcup_{z\le w} \Inc_{w,z}^\la$; given a vertex, the corresponding lex-increasing sequence
is obtained by reading the edge labels going from the vertex to the root.
We denote the $\la$-hyperplane $(\alpha,k)$ with coroot $\alpha=c_0\alpha_0^\vee+c_1\alpha_1^\vee+c_2\alpha_2^\vee$ 
by $\frac{(k| c_0,c_1,c_2)}{\ip{\alpha}{\la}}$, omitting the denominator if it is $1$.
The tree has $8$ vertices, which agrees with the 8 terms in Example \ref{X:LSup}.

Consider the rightmost leaf vertex in the tree. It is labeled $z=s_1s_2$,
and it defines the sequence $H=(h_1,h_2)\in \Inc^\la_{w,z}$ where $h_1=(0|0,1,0)$ and $h_2=1/3(2|1,2,2)$.

Let $\beta=\alpha_0^\vee+2\alpha_1^\vee+2\alpha_2^\vee$
and $\beta^\vee=\alpha_0+2\alpha_1+2\alpha_2$. Using the notation
$\alpha(ijk)$ for the element $i\alpha_0+j\alpha_1+k\alpha_2$, 
the weight of $H\in\Inc_{w,z}^\la$ is computed by
the right hand side of \eqref{E:incweight}:
\begin{align*}
\wt(H) &= s_1 s_2 s_1 t_{2\beta^\vee} s_\beta (\La_0+\La_1) \\
  &= s_1 s_2 s_1 t_{2\beta^\vee} (\La_0+\La_1-3\beta^\vee) \\
  &= s_1 s_2 s_1 (\La_0+\La_1-\alpha(122)) \\
  &= s_1 s_2 (\La_0+\La_1-\alpha(122)) \\
  &= s_1 (\La_0+\La_1-\alpha(121)) \\
  &= \La_0+\La_1 -\alpha(111).
\end{align*}
Let us compute the LS path in $\Pu^\la_{w,z}$
corresponding to $H\in\Inc_{w,z}^\la$. We use the notation of Section \ref{ls-nonrec}.
We have $m=2$, $(b_1,b_2,b_3)=(0,2/3,1)$,
$I_1 = \{h_1\}$, $I_2=\{h_2\}$,
$z_0=z=s_1s_2$, $z_1=s_1s_2s_1$, and $z_2=s_1s_2s_1s_\beta=s_0s_1s_2s_1$.
Recalling that $s_2\la=\la$, the directions of the LS path are given (in reverse order) by
$z_1\la = s_2s_1\la$ and $z_2\la = s_0s_2s_1\la$. The relative
distances traversed in these directions are given by $2/3-0$ and $1-2/3$, respectively. 
This is the path $p_1$ of Example \ref{X:LSup}. To check that $p_1\in \Pu_{w,z}^\la$,
we calculate the lifts $\up(s_1s_2,s_2s_1W_\la) = s_1s_2s_1=z_1$
and $\up(s_1s_2s_1,s_0s_2s_1W_\la)= s_0s_1s_2s_1=z_2=w$.
\end{ex}

\begin{rem} For the Chevalley rule for dominant weights, the alcove model is more efficient
than the LS path model. The alcove model uses only the computation of Bruhat
cocovers and comparisons of $\la$-hyperplanes with no wasted effort (if the
set of $\la$-hyperplanes adapted to cocovers of a given Weyl group element are remembered).
In contrast, the LS path formula requires
the generation of the entire Demazure crystal just to obtain the paths $p\in \Ta^\la$ 
with a specified initial direction $w\lambda$.
Then, for each such path $p$, one must compute the possible values of $z$ such that $\up(z,p)=w$.
For a fixed $p$, this can be done by starting with $w$, by considering which $w'$ in the second step (i.e., coset) of $p$ make the lift of the initial step $\sigma$ equal to $w$, that is, $\up(w',\sigma)=w$, and by continuing recursively. Thus, one needs to compute 
many Deodhar lifts, which are given by a non-trivial recursive procedure \cite{D}.
\end{rem}

\subsection{An antidominant example}
With $\geh$, $\la$, and $w$ fixed as above,
we compute the coefficients $b^w_{z,-\la}$ for all $z\le w$, again both by LS paths
and by the alcove model.

\begin{ex}\label{X:LSdown} We illustrate the
antidominant weight LS path Chevalley formula \eqref{comm1}. 

\begin{figure}
\begin{align*}
\begin{diagram}
\node{s_2} \arrow{s,t,->}{} \arrow{e,t,->}{} \node{s_1s_2} \arrow{s,t,->}{} \arrow{e,t,->}{} \node{s_1s_2s_1}
\arrow{s,t,->}{} \\
\node{s_0s_2} \arrow{e,t,!}{} \node{s_1s_2} \arrow{s,t,->}{} \arrow{e,t,!}{}
\node{s_1s_2s_1} \arrow{s,t,->}{} \\
\node[2]{s_0s_1s_2} \arrow{e,t,!}{} \node{s_1s_2s_1} \arrow{s,t,->}{} \\
\node[3]{s_0s_1s_2s_1}
\end{diagram}
\end{align*}
\caption{Computations of $\down$} 
\label{F:down}
\end{figure}

The entire set of 9 LS paths given in Figure \ref{F:XdomLS} must be used.
These paths form the Demazure crystal of lowest weight $w\lambda$. 
In Figure \ref{F:down}, in the position of each path $p$ in Figure \ref{F:XdomLS},
we place the element $\down(w,p)$, where we recall that $w=s_0s_1s_2s_1$.
The coefficients $b^w_{z,-\la}$ can be read from this diagram. For example,
if $z=s_1s_2$ then the coefficient is $e^{-\wt(q_1)}+e^{-\wt(q_2)}$, where $q_1=(s_1\la)$ and 
$q_2=((1/2) s_0s_1\la, (1/2) s_1\la)$
are the top two paths in the second column of the diagram in Figure \ref{F:XdomLS}.
We have $\wt(q_1)=\la-\alpha_1$ and $\wt(q_2)=\la-\alpha_0-\alpha_1$.
\end{ex}

\begin{ex}\label{X:alcantidom} We now work out the
antidominant weight alcove model Chevalley rule \eqref{chevantidom}.

\begin{figure}
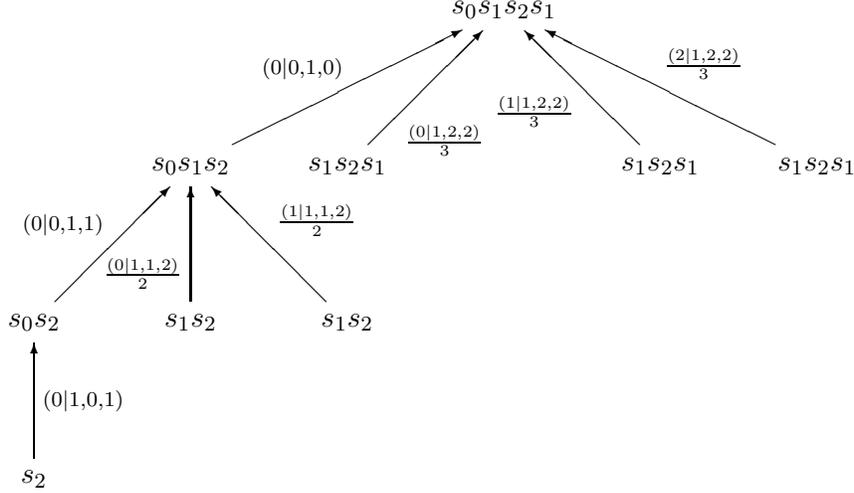

\begin{align*}
\begin{diagram}
\node[4]{s_0s_1s_2s_1}  \\
\node[2]{s_0s_1s_2} \arrow{ene,t,->}{(0|0,1,0)} \node{s_1s_2s_1}  \arrow{ne,b,1}{\frac{(0|1,2,2)}{3}}
\node[2]{s_1s_2s_1} \arrow{nw,b,2}{\frac{(1|1,2,2)}{3}} \node{s_1s_2s_1}  \arrow{wnw,t,->}{\frac{(2|1,2,2)}{3}} \\
\node{s_0s_2} \arrow{ne,t}{(0|0,1,1)} \node{s_1s_2} \arrow{n,t,1}{\frac{(0|1,1,2)}{2}}
\node{s_1s_2} \arrow{nw,t}{\frac{(1|1,1,2)}{2}} \\
\node{s_2}\arrow{n,b,->}{(0|1,0,1)}
\end{diagram}
\end{align*}
\caption{Antidominant tree}
\label{F:anti}
\end{figure}

As in the dominant case we create a suitable tree; see Figure \ref{F:anti}.
This tree has nine vertices, as there are 9 LS paths in Example \ref{X:LSdown}.
As before, given a vertex in the tree, we consider the path from the vertex up to the root.
Consider the sole vertex labeled $z=s_0 s_2$. The path to the root is given by
$H=(h_1>h_2)\in \Dec_{w,z}^\la$ where $h_1=(0|0,1,1)$ and $h_2=(0|0,1,0)$.
We now find the corresponding LS path in $\Pd^\la_{w,z}$ where $z=s_1s_2$. 
Since both $h_1$ and $h_2$ have relative coheight $1$, in the notation of
Section \ref{S:lexdecan}, we have $m=1$, $b_2=1$, $w_1=z=s_0s_2$, $w_2=w=s_0s_1s_2s_1$,
and $\sigma_1 = s_0s_2W_\la = s_0 W_\la$. The path is given by the single vector $(s_0\la)$.
Its endpoint is $s_0\la$. We have
$\down(s_0s_1s_2s_1, s_0W_\la) = s_0s_2=z$. The required saturated Bruhat chain in the interval
$[s_0s_2,s_0s_1s_2s_1]$ is uniquely specified by having increasing coroot labels with respect to
the reflection order on $\Phi^{\vee +}$ dual to the one induced by $0<1<2$.

We compute $\widetilde{\wt}(H)$. Let $\gamma=\alpha_1^\vee+\alpha_2^\vee$.
Then $\gamma=\alpha_1+\alpha_2$, $\ip{\gamma}{\la}=1$, and 
\begin{align*}
  \widetilde{\wt}(H) &= s_0 s_2 t_{\gamma^\vee} s_{\gamma} t_{\alpha_1} s_{1} (\La_0+\La_1) \\
  &= s_0 s_2 t_{\gamma^\vee} s_\gamma t_{\alpha_1} (\La_0+\La_1-\alpha_1) \\
  &= s_0 s_2 t_{\gamma^\vee} s_\gamma (\La_0+\La_1) \\
  &= s_0 s_2 t_{\gamma^\vee} (\La_0+\La_1-\gamma^\vee) \\
  &= s_0 s_2 (\La_0+\La_1) \\
  &= s_0 (\La_0+\La_1)
\end{align*}
which agrees with the weight of the LS path.

\end{ex}

\subsection{The affine Grassmannian} Let us consider the $K$-theory of the affine Grassmannian 
$Gr_{SL_n}$ of type $A_{n-1}$. It is known that we can index the vertices of the crystal 
of highest weight $\Lambda_0$ by the {\em $n$-restricted partitions} in the 
Misra-Miwa model \cite{mm}. Note the concepts of roof and base of an $n$-restricted 
partition $J$, which are defined based on some subtle combinatorial constructions \cite{akt}. 
The roof and ceiling lemmas in \cite{akt} state that ${\rm roof}(J)$ and ${\rm base}(J)$ 
are essentially the initial and final directions of the corresponding LS path, respectively. 
On another hand, it is known that the Schubert classes in the $K$-theory of $Gr_{SL_n}$ 
are indexed by {\em stable $n$-restricted (or core) partitions}, see \cite{LSS}. 
These correspond to lowest coset representatives in $W/W_{\Lambda_0}$, where $W$ is the 
affine symmetric group and $W_{\Lambda_0}$ is the symmetric group (on $n$ letters). 

Consider the Chevalley product $[L^{\Lambda_0}]\, [\OO_{X_I}]$, where $I$ is a
stable $n$-restricted partition. By the roof and ceiling lemmas, the $K$-Chevalley formula \eqref{chdom} becomes
\begin{equation}\label{affchev}
[L^{\Lambda_0}]\, [\OO_{X_I}]=\sum_{J\::\:{\rm base}(J)\ge I} e^{{\rm wt}(J)} [\OO_{X_{{\rm roof}(J)}}]\,;
\end{equation}
here the summation is over the corresponding $n$-restricted partitions, and for the 
definition of the weight ${\rm wt}(J)$ we also refer to \cite{akt}. Note that all the 
Deodhar lifts in \eqref{chdom} are trivial in this case.

Similarly, one can derive a formula for $[\OO_{X_{s_0}}]\, [\OO_{X_I}]$ based on 
\eqref{chantidom} and \eqref{divisor}. More generally, $\Lambda_0$ and $s_0$ can be replaced by 
any $\Lambda_i$ and $s_i$, using the corresponding notions in \cite{akt}. 

It would be interesting to connect the model of $n$-restricted partitions and the alcove model. 
As the analogues of the initial and final direction of an LS path are easy to read off 
in the alcove model, such a connection would lead to a more transparent construction 
of roof and base (the current construction is highly non-transparent). 


\end{document}